\documentclass[a4paper,twoside]{article}

\def\shorttitle{Weak Unique Continuation Property for Fractional Diffusion
Equations}
\def\shortauthor{D. Jiang, Z. Li, Y. Liu and M. Yamamoto}

\usepackage{amsfonts}
\usepackage{amsmath}
\usepackage{amssymb}
\usepackage{amscd}
\usepackage{amsbsy}
\usepackage{amsthm}
\usepackage{bm}
\usepackage{empheq}
\usepackage{graphicx}
\usepackage{psfrag}
\usepackage{color}
\usepackage{cite}
\usepackage{multicol}
\usepackage[normalem]{ulem}

\allowdisplaybreaks[4]

\newfont{\myfnt}{cmssi10 scaled 1440}
\numberwithin{equation}{section}
\catcode`@=11
\def\ps@nk{\def\@oddhead{\vbox{\hbox to \hsize{\pic \footnotesize \it \shorttitle
\hfill \rm \thepage} \vspace{1mm} \vspace*{-2mm}}}
\def\@evenhead{\vbox{\hbox to \hsize{\pic \footnotesize \rm \thepage \hfill \it \shortauthor}
\vspace{1mm} \vspace*{-2mm}}}
\def\@oddfoot{} \def\@evenfoot{}}
\def\ps@first{\def\@oddhead{\vbox{\hbox to \hsize{\pic \footnotesize
} \break}}
\def\@oddfoot{} \def\@evenfoot{}}
\newtheoremstyle{thmstyle}
  {6pt}
  {6pt}
  {\it}
  {}
  {\bf}
  {}
  {.5em}
  {}
\newtheoremstyle{remstyle}
  {6pt}
  {6pt}
  {\rm}
  {}
  {\bf}
  {}
  {.5em}
  {}
\def\Section#1{\Sec{\large #1} \setcounter{equation}{0} \vskip -6mm \indent}
\def\Sec{\@Startsection{section}{1}{\z@}
                                   {-3.5ex \@plus -1ex \@minus -.2ex}%
                                   {2.3ex \@plus.2ex}%
                                   {\normalfont\large\bfseries\boldmath}}
\def\@Startsection#1#2#3#4#5#6{%
  \if@noskipsec \leavevmode \fi
  \par
  \@tempskipa #4\relax
  \@afterindenttrue
  \ifdim \@tempskipa <\z@
    \@tempskipa -\@tempskipa \@afterindentfalse
  \fi
  \if@nobreak
    \everypar{}%
  \else
    \addpenalty\@secpenalty\addvspace\@tempskipa
  \fi
  \@ifstar
    {\@ssect{#3}{#4}{#5}{#6}}%
    {\@dblarg{\@Sect{#1}{#2}{#3}{#4}{#5}{#6}}}}
\def\@Sect#1#2#3#4#5#6[#7]#8{%
  \ifnum #2>\c@secnumdepth
    \let\@svsec\@empty
  \else
    \refstepcounter{#1}%
    \protected@edef\@svsec{\@seccntformat{#1}\relax}%
  \fi
  \@tempskipa #5\relax
  \ifdim \@tempskipa>\z@
    \begingroup
      #6{%
          \@hangfrom{\hskip #3\relax\@svsec \hskip -2.5mm}%
          \interlinepenalty \@M #8\@@par}
    \endgroup
    \csname #1mark\endcsname{#7}%
    \addcontentsline{toc}{#1}{%
      \ifnum #2>\c@secnumdepth \else
        \protect\numberline{\csname the#1\endcsname}%
      \fi
      #7}%
  \else
    \def\@svsechd{%
      #6{\hskip #3\relax
      \@svsec #8}%
      \csname #1mark\endcsname{#7}%
      \addcontentsline{toc}{#1}{%
        \ifnum #2>\c@secnumdepth \else
          \protect\numberline{\csname the#1\endcsname}%
        \fi
        #7}}%
  \fi
  \@xsect{#5}}
\renewenvironment{abstract}{%
        \small
        \quotation
         \noindent {\bfseries \abstractname } }%
      {\if@twocolumn\else\endquotation\fi}

\def\Subsec{\@StartSubsection{subsection}{2}{\z@}%
                                     {-3.25ex\@plus -1ex \@minus -.2ex}%
                                     {1.5ex \@plus .2ex}%
                                     {\normalfont\normalsize\bfseries\boldmath}}
\def\@StartSubsection#1#2#3#4#5#6{%
  \if@noskipsec \leavevmode \fi
  \par
  \@tempskipa #4\relax
  \@afterindenttrue
  \ifdim \@tempskipa <\z@
    \@tempskipa -\@tempskipa \@afterindentfalse
  \fi
  \if@nobreak
    \everypar{}%
  \else
    \addpenalty\@secpenalty\addvspace\@tempskipa
  \fi
  \@ifstar
    {\@ssect{#3}{#4}{#5}{#6}}%
    {\@dblarg{\@SubSect{#1}{#2}{#3}{#4}{#5}{#6}}}}
\def\@SubSect#1#2#3#4#5#6[#7]#8{%
  \ifnum #2>\c@secnumdepth
    \let\@svsec\@empty
  \else
    \refstepcounter{#1}%
    \protected@edef\@svsec{\@seccntformat{#1}\relax}%
  \fi
  \@tempskipa #5\relax
  \ifdim \@tempskipa>\z@
    \begingroup
      #6{%
          \@hangfrom{\hskip #3\relax\@svsec\hskip -1.5mm}%
          \interlinepenalty \@M #8\@@par}
    \endgroup
    \csname #1mark\endcsname{#7}%
    \addcontentsline{toc}{#1}{%
      \ifnum #2>\c@secnumdepth \else
        \protect\numberline{\csname the#1\endcsname}%
      \fi
      #7}%
  \else
    \def\@svsechd{%
      #6{\hskip #3\relax
      \@svsec #8}%
      \csname #1mark\endcsname{#7}%
      \addcontentsline{toc}{#1}{%
        \ifnum #2>\c@secnumdepth \else
          \protect\numberline{\csname the#1\endcsname}%
        \fi
        #7}}%
  \fi
  \@xsect{#5}}
\def\list#1#2{\ifnum \@listdepth >5\relax \@toodeep \else \global
\advance \@listdepth\@ne \fi \rightmargin \z@ \listparindent\z@
\itemindent\z@ \csname @list\romannumeral\the\@listdepth\endcsname
\def\@itemlabel{#1}\let\makelabel\@mklab \@nmbrlistfalse #2\relax
\@trivlist \parskip 0pt \parindent\listparindent \advance \linewidth
-\rightmargin \advance\linewidth -\leftmargin \advance\@totalleftmargin
\leftmargin \parshape \@ne \@totalleftmargin \linewidth \ignorespaces}
\renewcommand{\@makecaption}[2]{\begin{center}#1. #2\end{center}}
\catcode`@=12 
\theoremstyle{thmstyle}
\newtheorem{thm}{\indent Theorem}[section]
\newtheorem{lem}[thm]{\indent Lemma}

\newtheorem{defi}[thm]{\indent Definition}

\newtheorem{prob}[thm]{\indent Problem}
\theoremstyle{remstyle}
\newtheorem{rem}[thm]{\indent Remark}
\newtheorem{algo}[thm]{\indent Algorithm}
\newtheorem{ex}[thm]{\indent Example}

\newsavebox{\mygraphic}
\def\pic{\begin{picture}(0,0) \put(-210,-1250){\usebox{\mygraphic}} \end{picture}}
\newfont{\HUGEbf}{cmbx10 scaled 3500}
\definecolor{gray}{rgb}{0.9,0.9,0.9}
\def\thebibliography#1{\section*{\bf \large References}
\list{[\arabic{enumi}]} {\settowidth \labelwidth{[#1]} \leftmargin
\labelwidth \advance \leftmargin \labelsep \usecounter{enumi}}
\def\newblock{\hskip .11em plus .33em minus .07em} \footnotesize \sloppy \clubpenalty
4000 \widowpenalty 4000 \sfcode`\.=1000 \relax}

\def\BC{\mathbb C}

\def\BR{\mathbb R}
\def\cA{\mathcal A}
\def\cD{\mathcal D}

\def\cK{\mathcal K}
\def\cL{\mathcal L}
\def\cP{\mathcal P}
\def\cR{\mathcal R}

\def\rd{\mathrm d}
\def\rRe{\mathrm{Re}}

\def\e{\mathrm e}
\def\err{\mathrm{err}}
\def\rand{\mathrm{rand}}

\def\true{\mathrm{true}}
\def\Ga{\Gamma}

\def\Om{\Omega}
\def\al{\alpha}
\def\be{\beta}
\def\ga{\gamma}
\def\de{\delta}
\def\ep{\epsilon}
\def\ve{\varepsilon}
\def\te{\theta}
\def\ze{\zeta}
\def\ka{\kappa}
\def\la{\lambda}

\def\vp{\varphi}
\def\om{\omega}
\def\f{\frac}
\def\nb{\nabla}
\def\ov{\overline}
\def\pa{\partial}
\def\wh{\widehat}
\def\wt{\widetilde}
\def\tri{\triangle}
\def\beqnx{\begin{eqnarray*}} \def\eqnx{\end{eqnarray*}}

\theoremstyle{definition}

\numberwithin{equation}{section}

\title{\Large\bf\boldmath Weak Unique Continuation Property and a Related
Inverse Source Problem for Time-Fractional Diffusion-Advection Equations}

\author{\large Daijun JIANG$^\dag$\quad Zhiyuan LI$^\ddag$\quad Yikan
LIU$^\ddag$ \quad Masahiro YAMAMOTO$^\ddag$}

\date{}

\begin{document}

\maketitle

\thispagestyle{first}
\renewcommand{\thefootnote}{\fnsymbol{footnote}}

\footnotetext{\hspace*{-5mm} \begin{tabular}{@{}r@{}p{14cm}@{}} &
Manuscript last updated: \today.\\
$^\dag$ & School of Mathematics and Statistics $\&$ Hubei Key Laboratory of
Mathematical Sciences, Central China Normal University, Wuhan 430079, China.
E-mail: jiangdaijun@mail.ccnu.edu.cn.\\
$^\ddag$ & Graduate School of Mathematical Sciences, The University of Tokyo,
3-8-1 Komaba, Meguro-ku, Tokyo 153-8914, Japan. E-mail: zyli@ms.u-tokyo.ac.jp,
ykliu@ms.u-tokyo.ac.jp, myama@ms.u-tokyo.ac.jp
\end{tabular}}

\renewcommand{\thefootnote}{\arabic{footnote}}

\begin{abstract}
In this paper, we first establish a weak unique continuation property for
time-fractional diffusion-advection equations. The proof is mainly based on the
Laplace transform and the unique continuation properties for elliptic and
parabolic equations. The result is weaker than its parabolic counterpart in the
sense that we additionally impose the homogeneous boundary condition. As a
direct application, we prove the uniqueness for an inverse problem on
determining the spatial component in the source term in by interior
measurements. Numerically, we reformulate our inverse source problem as an
optimization problem, and propose an iteration thresholding algorithm. Finally,
several numerical experiments are presented to show the accuracy and efficiency
of the algorithm.

\vskip 4.5mm

\noindent\begin{tabular}{@{}l@{ }p{10cm}} {\bf Keywords } & fractional
diffusion equation, weak unique continuation, inverse source problem, iterative
thresholding algorithm
\end{tabular}

\vskip 4.5mm


\end{abstract}

\baselineskip 14pt

\setlength{\parindent}{1.5em}

\setcounter{section}{0}

\Section{Introduction}

Let $\Om\subset\BR^d$ be an open bounded domain with a sufficiently smooth
boundary (e.g., of $C^2$-class) and $T>0$. Let $m\in\{1,2,\ldots\}$ and
$\al_j,q_j$ ($j=1,2,\ldots,m$) be positive constants such that
$1>\al_1>\al_2>\cdots>\al_m>0$. By $\pa_t^{\al_j}$ we denote the Caputo
derivative (see, e.g., \cite[\S 2.4.1]{P99})
\[
\pa_t^{\al_j}g(t)
:=\f1{\Ga(1-\al_j)}\int_0^t\f{g'(\tau)}{(t-\tau)^{\al_j}}\,\rd\tau,
\]
where $\Ga(\,\cdot\,)$ stands for the Gamma function. For
$(x,t)\in Q:=\Om\times(0,T)$, we define the operator
\begin{equation}\label{eq-def-P}
\cP u(x,t):=\sum_{j=1}^mq_j\pa_t^{\al_j}u(x,t)+\cA u(x,t)+B(x)\cdot\nb u(x,t).
\end{equation}
Here $\cA$ is a symmetric second-order elliptic operator which will be defined
at the beginning of Section \ref{sec-pre}, and
$B(x)=(b_1(x),b_2(x),\ldots,b_d(x))$. Without loss of generality, we set
$q_1=1$. In this paper, we investigate the following initial-boundary value
problem for the time-fractional diffusion-advection equation
\begin{equation}\label{eq-ibvp-u}
\begin{cases}
\cP u=F & \mbox{in }Q,\\
u=a & \mbox{in }\Om\times\{0\},\\
u=0\mbox{ or }\pa_\cA u=0 & \mbox{on }\pa\Om\times(0,T),
\end{cases}
\end{equation}
where $\pa_\cA u$ denotes the normal derivative associated with the elliptic
operator $\cA$. The conditions on the initial data $a$, the source term $F$,
coefficients involved in $\cP$ and the definitions of $\pa_\cA$ will be
specified later in Section \ref{sec-pre}.

In various forms and generalities, the time-fractional parabolic operator $\cP$
in \eqref{eq-def-P} has gained increasing popularity among mathematicians
within the last few decades, owing to its applicability in describing the
anomalous diffusion phenomena in highly heterogeneous media (see
\cite{AG92,HH98} and the references therein). The fundamental theory for the
single-term (i.e., $m=1$) case of \eqref{eq-def-P} was established around the
early 2010s, represented by the maximum principle proved in Luchko \cite{L10}
and the well-posedness, analyticity and asymptotic behavior proved in Sakamoto
and Yamamoto \cite{SY11a}. Thereafter, most of the properties were parallelly
generalized to the multi-term case (i.e., $m>1$) in \cite{L11,LLY15,LY13}, and
especially the maximum principle was recently improved to stronger ones in
\cite{LRY16,L15}. Meanwhile, corresponding numerical methods have also been
well-developed and we refer e.g.\! to \cite{JLZ13,JLLZ15}. In contrast to the
usual parabolic equations characterized by the exponential decay in time and
Gaussian profile in space, it reveals that the fractional diffusion equations
driven by $\cP$ possess properties of slow decay in time and long-tailed
profile in space. Nevertheless, we notice that most of the existing literature
only treated the symmetric elliptic operator (i.e., $B\equiv0$ in
\eqref{eq-def-P}), in which the existence of eigensystem provides convenience
for the argument.

Other than the above mentioned aspects, the unique continuation property is
also one of the remarkable characterizations of parabolic equations, which
asserts the vanishment of a solution to a homogeneous problem in an open subset
implies its vanishment in the whole domain (see, e.g., \cite{SS87}). The unique
continuation property is not only important by itself, but also significant in
its applications to many related control and inverse problems. However, the
publications on its generalization to fractional diffusion equations are rather
limited to the best of the authors' knowledge. For the special half-order
fractional diffusion equation (i.e., $m=1$, $\al_1=\f12$ and $\cA=-\tri$ in
\eqref{eq-def-P}), the unique continuation property was proved in Xu, Cheng and
Yamamoto \cite{XCY11} for $d=1$ and Cheng, Lin and Nakamura \cite{CLN13} for
$d=2$ via Carleman estimates for the operator $\pa_t-\tri^2$. For a general
fractional order in the $(0,1)$ interval, Lin and Nakamura \cite{LN16} recently
obtained a unique continuation property by using a newly established Carleman
estimate based on calculus of pseudo-differential operators. We notice that the
conclusion in \cite{LN16} requires the homogeneous initial condition, which
possibly roots in the memory effect of time-fractional diffusion equations.

Regarding the unique continuation property, the first focus of this paper is
the investigation of the following problem.

\begin{prob}\label{prob-wuc}
Let $u$ be the solution of \eqref{eq-ibvp-u}, where the source term $F=0$. Then
does $u=0$ in some open subset of $Q=\Om\times(0,T)$ implies $u\equiv0$ in $Q$
under certain conditions?
\end{prob}

In Theorem \ref{thm-wuc}, we will give an affirmative answer to this problem.
Compared with the existing literature, we formulate the problem on the more
general time-fractional parabolic operator $\cP$ with non-symmetric elliptic
part in space. Meanwhile, we allow non-vanishing initial data at the cost of
the homogeneous Dirichlet or Neumann boundary condition.

On the other hand, parallelly with the intensive attention paid to forward
problems for time-fractional diffusion equations, there are also rapidly
growing publications on the related inverse problems with various combinations
of unknown functions and observation data. Here we do not intend to give a full
list of bibliographies, but just mention
\cite{CNYY09,LZJY13,MY13,Z16,LY15,LIY16} and the references therein for
readers' curiosity. Nevertheless, it turns out that the majority of them
concentrate on coefficient inverse problems. In contrast, the study on inverse
source problems is far from satisfactory and mainly restricts to several
special cases due to the lack of specified techniques. In the one-dimensional
case, Zhang and Xu \cite{ZX11} proved the uniqueness for determining a
time-independent source term by the partial boundary data, and a conditional
stability for the recovery of the spatial component in the source term was
proved for the half-order case in Yamamoto and Zhang \cite{YZ12}. With the
final overdetermining data, Sakamoto and Yamamoto \cite{SY11b} showed the
generic well-posedness for reconstructing the spatial component. Similarly to
the situation of the forward problems reviewed above, it reveals that almost
all papers treating the related inverse problems also rely heavily on the
symmetry of the involved elliptic operator, regardless of the practical
importance of the non-symmetric case.

Keeping the above points in mind, we are also interested in studying the
following inverse source problem, which is the second focus of this paper.

\begin{prob}\label{prob-isp}
Let $u$ be the solution of \eqref{eq-ibvp-u}, where the initial data $a=0$ and
the source term takes the form of separated variables, namely
$F(x,t)=f(x)\,\mu(t)$. Provided that the temporal component $\mu(t)$
$(0\le t\le T)$ is known, can we uniquely determine the spatial component
$f(x)$ $(x\in\Om)$ by the partial interior observation of $u$ in some open
subset of $Q=\Om\times(0,T)$ under certain conditions?
\end{prob}

Theorem \ref{thm-isp} answers this problem affirmatively. Obviously, the above
problem is closely related to Problem \ref{prob-wuc} in the sense that both are
concerned with the partial interior information of the solution. Practically,
the formulation of Problem \ref{prob-isp} is applicable in the determination of
the space distribution $f$ modeling the contaminant source, where the anomalous
diffusion phenomena is described by \eqref{eq-ibvp-u} and the time evolution
$\mu$ of the contaminant is known in advance. As far as the authors know, the
above problem has not yet been considered in form of the generalized
time-fractional parabolic operator $\cP$.

By restricting the open subset in Problems \ref{prob-wuc}--\ref{prob-isp} as a
cylindrical subdomain, first we will give an affirmative answer to Problem
\ref{prob-wuc} in two cases, that is, either the multi-term fractional
diffusion equation without an advection term or the single-term one with an
advection term. The statement concluded in Theorem \ref{thm-wuc} will be called
as the weak unique continuation property because we impose the homogeneous
Dirichlet or Neumann boundary condition, which is absent in the usual parabolic
prototype. As a direct application, the uniqueness for Problem \ref{prob-isp}
can be immediately proved with the aid of a fractional version of Duhamel's
principle. For the numerical reconstruction, we reformulate Problem
\ref{prob-isp} as an optimization problem with Tikhonov regularization. After
the derivation of the corresponding variational equation, we can characterize
the minimizer by employing the associated backward fractional diffusion
equation, which results in an efficient iterative method.

The remainder of this paper is organized as follows. Preparing all necessities
about the weak solution of \eqref{eq-ibvp-u}, in Section \ref{sec-pre} we state
the main results answering Problems \ref{prob-wuc} and \ref{prob-isp} in
Theorems \ref{thm-wuc} and \ref{thm-isp}, respectively. Then Section
\ref{sec-proof} is devoted to the proofs of the above theorems. In Section
\ref{sec-ita}, we propose the iterative thresholding algorithm for the
numerical treatment of our inverse source problem, followed by several
numerical examples illustrating the performance of the proposed method in
Section \ref{sec-numer}. As technical details, we provide the proofs for the
well-posedness of the weak solutions of \eqref{eq-ibvp-u} in Appendix
\ref{sec-app}.

\Section{Preliminaries and Main Results}\label{sec-pre}

In this section, we first set up notations and terminologies, and review some
of standard facts on the fractional calculus. Let $L^2(\Om)$ be a usual
$L^2$-space with the inner product $(\,\cdot\,,\,\cdot\,)$ and $H_0^1(\Om)$,
$H^2(\Om)$, etc.\! denote the usual Sobolev spaces. Especially, for
$\be\in(0,1)$ we define the fractional Sobolev space $H^\be(0,T)$ in time (see
Adams \cite{A75}). The elliptic operator $\cA$ is defined for
$\psi\in\cD(\cA):=\{\psi\in H^2(\Om);\,\psi=0\mbox{ on }\pa\Om\}$ or
$\{\psi\in H^2(\Om);\,\pa_\cA\psi=0\mbox{ on }\pa\Om\}$ as
\[
\cA\psi(x):=-\sum_{i,j=1}^d\pa_j(a_{ij}(x)\pa_i\psi(x))+c(x)\psi(x),
\]
where $\pa_\cA\psi(x):=\sum_{i,j=1}^da_{ij}(x)\nu_i(x)\pa_j\psi(x)$ and
$\nu(x)=(\nu_1(x),\ldots,\nu_d(x))$ denotes the outward unit normal vector to
$\pa\Om$ at $x\in\pa\Om$. Here we assume $a_{ij}=a_{ji}\in C^1(\ov\Om)$
($1\le i,j\le d$), $c\in L^\infty(\Om)$ and there exists a constant $\ka>0$
such that
\[
\sum_{i,j=1}^da_{ij}(x)\xi_i\xi_j\ge\ka\sum_{i=1}^d\xi_i^2,
\quad\forall\,x\in\ov\Om,\ \forall\,(\xi_1,\ldots,\xi_d)\in\BR^d.
\]
When the zeroth order coefficient $c\ge0$ in $\Om$, we introduce the
eigensystem $\{(\la_n,\vp_n)\}_{n=1}^\infty$ of $\cA$ such that
$0\le\la_1<\la_2\le\cdots$, $\la_n\to\infty$ ($n\to\infty$) and
$\{\vp_n\}\subset\cD(\cA)$ forms a complete orthonormal basis of $L^2(\Om)$.
Considering the possibility of $\la_1=0$, we define $\wt\cA:=\cA+1$. Then the
corresponding eigenvalues $\wt\la_n:=\la_n+1$ are all strictly positive, and
the fractional power $\wt\cA^\ga$ is defined for $\ga\in\BR$ (e.g., \cite{P83})
as
\[
\wt\cA^\ga\psi=\sum_{n=1}^\infty\wt\la_n^\ga(\psi,\vp_n)\vp_n,
\]
where
\[
\psi\in\cD(\wt\cA^\ga):=\left\{\psi\in L^2(\Om);
\sum_{n=1}^\infty\wt\la_n^{2\ga}|(\psi,\vp_n)|^2<\infty\right\}
\]
and $\cD(\wt\cA^\ga)$ is a Hilbert space with the norm
\[
\|\psi\|_{\cD(\wt\cA^\ga)}
=\left(\sum_{n=1}^\infty\left|\wt\la_n^\ga(\psi,\vp_n)\right|^2\right)^{1/2}.
\]
On the other hand, the first order coefficient $B=(b_1,\ldots,b_d)$ in the
operator $\cP$ is assumed to be in $(L^\infty(\Om))^d$.

By $J_{0+}^\al$ we denote the Riemann-Liouville integral operator, which is
defined by
\[
J_{0+}^\al g(t):=\f1{\Ga(\al)}\int_0^t\f{g(\tau)}{(t-\tau)^{1-\al}}\,\rd\tau,
\quad\al>0.
\]
Then the Caputo derivative $\pa_t^\al$ can be rephrased as
$\pa_t^\al g(t)=J_{0+}^{1-\al}\f\rd{\rd t}g(t)$. Parallelly, we define the
backward Riemann-Liouville integral operator $J^\al_{T-}$ by
\[
J^\al_{T-}g(t)=\f1{\Ga(\al)}\int_t^T \f{g(\tau)}{(\tau-t)^{1-\al}}\,\rd\tau,
\]
and the backward Riemann-Liouville fractional derivative by
$D_t^\al g(t):=-\f\rd{\rd t}J^{1-\al}_{T-}g(t)$.

First we state the well-posedness result for the homogeneous case of the
initial-boundary value problem \eqref{eq-ibvp-u}.

\begin{lem}\label{lem-fp-homo}
Assume $F=0$, $a\in L^2(\Om)$ and let $\ga\in(3/4,1)$ be a given constant. Then
there exists a unique solution
$u\in C((0,T];\cD(\wt\cA^\ga))\cap C([0,T];L^2(\Om))$ to the problem
\eqref{eq-ibvp-u}. Moreover, the solution
$u:(0,T)\longrightarrow\cD(\wt\cA^\ga)$ is analytic and can be analytically
extended to $(0,\infty)$. Further, there exists a constant
$C=C(d,\al_j,q_j,\cA,B,\ga)>0$ such that
\[
\|u(\,\cdot\,,t)\|_{\cD(\wt\cA^\ga)}
\le C\,\e^{CT}t^{-\al_1\ga}\|a\|_{L^2(\Om)},\quad 0<t<T.
\]
\end{lem}

\begin{rem}\label{rem-analy}
The proof of Lemma \ref{lem-fp-homo} is very similar to that of
\cite[Theorem 4.1]{LY13}, which only treated the homogeneous Dirichlet boundary
condition. Moreover, we point out that in the case of $B\equiv0$ and $c\ge0$,
the regularity of the solution can be improved to $C((0,T];\cD(\wt\cA\,))$.
\end{rem}

Now we turn to the inhomogeneous problem, i.e., $a=0$ and $F\ne0$. Since
\cite[Theorem 1.1]{GLY15} asserts the $H^\al(0,T;L^2(\Om))$ regularity of the
solution, we see that the initial value becomes delicate in the case of
$\al\le1/2$ because the time-regularity does not make sense pointwisely
anymore. Following the same line as that in \cite{GLY15}, we shall redefine the
weak solution to \eqref{eq-ibvp-u}.

\begin{defi}[Weak solution]\label{def-fp-weak}
Let $F\in L^2(Q)$. We say that $u$ is a weak solution to the initial-boundary
value problem \eqref{eq-ibvp-u} with $a=0$ if
\[
u\in L^2(0,T;\cD(\wt\cA\,)),\quad J_{0+}^{-\al_1}u\in L^2(Q),
\quad\cP u=F\mbox{ in }L^2(Q).
\]
Here $J_{0+}^{-\al_1}$ denotes the inverse operator of the Riemann-Liouville
integral operator $J_{0+}^{\al_1}$.
\end{defi}

In Definition \ref{def-fp-weak}, we should understand the Caputo derivative
$\pa_t^{\al_j}$ ($j=1,2,\ldots,m$) in the operator $\cP$ as the unique
extension of the operator $\pa^{\al_j}_t:C^\infty(0,T)\to L^2(0,T)$ to
$H^{\al_j}(0,T)$ according to \cite{GLY15}.

Within this framework, we can prove the following well-posedness result.

\begin{lem}[Well-posedness of Definition \ref{def-fp-weak}]\label{lem-fp-weak}
Let $a=0$ and $F\in L^2(Q)$. Then the initial-boundary value problem
\eqref{eq-ibvp-u} admits a unique weak solution
$u\in L^2(0,T;\cD(\wt\cA\,))\cap H^{\al_1}(0,T;L^2(\Om))$. Moreover, there
exists a constant $C>0$ such that
\[
\|u\|_{H^{\al_1}(0,T;L^2(\Om))}+\|u\|_{L^2(0,T;\cD(\wt\cA\,))}
\le C\|F\|_{L^2(Q)}.
\]
\end{lem}

The proof of the above lemma will be given in Appendix \ref{sec-app}.

By Lemma \ref{lem-fp-homo} and the unique continuation for elliptic and
parabolic equations, we can establish the weak type unique continuation for the
fractional parabolic equation.

\begin{thm}\label{thm-wuc}
Let $a\in L^2(\Om)$, $F=0$ and $u$ be the solution to \eqref{eq-ibvp-u}. Let
$\om\subset\Om$ be an arbitrarily chosen open subdomain. Then
\[
u=0\mbox{ in }\om\times(0,T)\quad\mbox{implies}
\quad u=0\mbox{ in }\Om\times(0,T)
\]
in either of the following two cases.

{\rm{\bf Case 1}}\ \ $m=1$, i.e., $\cP$ is a single-term time-fractional
parabolic operator.

{\rm{\bf Case 2}}\ \ $B\equiv0$ and $c\ge0$ in $\Om$, i.e., the first order
coefficient in $\cP$ vanishes and the zeroth order one is non-negative.
\end{thm}

Sakamoto and Yamamoto \cite{SY11a} proved Theorem \ref{thm-wuc} for the
symmetric single-term time-fractional diffusion equation by use of the
eigenfunction expansion and the unique continuation property for elliptic
equations. However, their method cannot work for the non-symmetric counterpart
because their argument relies heavily on the symmetry of the elliptic operator.

As an immediate application of the above property, we can give a uniqueness
result for Problem \ref{prob-isp}.

\begin{thm}\label{thm-isp}
Let $a=0$ and $F(x,t)=f(x)\,\mu(t)$, where $f\in L^2(\Om)$ and
$\mu\in C^1[0,T]$ with $\mu(0)\ne0$. Let $u$ be the solution to
\eqref{eq-ibvp-u} and $\om\subset\Om$ be an arbitrarily chosen open subdomain.
Then in either case in Theorem $\ref{thm-wuc}$, $u=0$ in $\om\times(0,T)$
implies $f=0$ in $\Om$.
\end{thm}

\Section{Proofs of Theorems \ref{thm-wuc} and \ref{thm-isp}}\label{sec-proof}

In this section, we give the proofs of Theorems \ref{thm-wuc} and
\ref{thm-isp}.

\begin{proof}[Proof of Theorem $\ref{thm-wuc}$]
According to our assumptions and Lemma \ref{lem-fp-homo}, the solution $u$ to
the initial-boundary value problem \eqref{eq-ibvp-u} can be analytically
extended from $(0,T)$ to $(0,\infty)$. For simplicity, we still denote the
extension by $u$. Then we arrive at the following initial-boundary value
problem
\begin{equation}\label{equ-u-infty}
\begin{cases}
\cP u =0 & \mbox{in }\Om\times(0,\infty),\\
u=a & \mbox{in }\Om\times\{0\},\\
u=0\mbox{ or }\pa_\cA u=0 & \mbox{on }\pa\Om\times(0,\infty),
\end{cases}
\end{equation}
and the condition $u=0$ in $\om\times(0,T)$ implies
\begin{equation}\label{eq-van}
u=0\quad\mbox{in}\quad\om\times(0,\infty)
\end{equation}
immediately. We divide the proof into the two cases described in Theorem
\ref{thm-wuc}.\medskip

{\bf Case 1}\ \ $m=1$. For simplicity, we write $\al=\al_1$. We perform the
Laplace transform (denoted by $\wh\cdot$\,) in \eqref{equ-u-infty} and use the
formula
\[
\wh{\pa_t^\al g}(s)=s^\al\,\wh g(s)-s^{\al-1}g(0+)
\]
to derive the transformed algebraic equation
\[
\begin{cases}
(\cA+s^\al)\wh u(x;s)+B(x)\cdot\nb\wh u(x;s)=s^{\al-1}a(x), & x\in\Om,\\
\wh u(x;s)=0\mbox{ or }\pa_\cA\wh u(x;s)=0, & x\in\pa\Om
\end{cases}
\]
with a parameter $s>s_1$, where $s_1>0$ is a sufficiently large constant.
Multiplying both sides of the above equation by $s^{1-\al}$ and setting
$\wh u_1(x;s):=s^{1-\al}\wh u(x;s)$, we then obtain the following boundary
value problem for an elliptic equation
\begin{equation}\label{equ-lap-v}
\begin{cases}
(\cA+s^\al)\wh u_1(x;s)+B(x)\cdot\nb\wh u_1(x;s)=a(x), & x\in\Om,\\
\wh u_1(x;s)=0\mbox{ or }\pa_\cA\wh u_1(x;s)=0, & x\in\pa\Om,
\end{cases}\quad s>s_1.
\end{equation}
On the other hand, let us consider an initial-boundary value problem for a
parabolic equation
\[
\begin{cases}
\pa_tu_2+\cA u_2+B\cdot\nb u_2=0 & \mbox{in }\Om\times(0,\infty),\\
u_2=a & \mbox{in }\Om\times\{0\},\\
u_2=0\mbox{ or }\pa_\cA u_2=0 & \mbox{on }\pa\Om\times(0,\infty).
\end{cases}
\]
Again, applying the Laplace transform yields
\[
\begin{cases}
(\cA+\eta)\wh u_2(x;\eta)+B(x)\cdot\nb\wh u_2(x;\eta)=a(x), & x\in\Om,\\
\wh u_2(x;\eta)=0\mbox{ or }\pa_\cA\wh u_2(x;\eta)=0, & x\in\pa\Om,
\end{cases}
\]
where the parameter $\eta>s_2$ and $s_2>0$ is a sufficiently large constant.
After the change of variable $\eta=s^\al$, we find
\[
\begin{cases}
(\cA+s^\al)\wh u_2(x;s^\al)+B(x)\cdot\nb\wh u_2(x;s^\al)=a(x), & x\in\Om,\\
\wh u_2(x;s^\al)=0\mbox{ or }\pa_\cA\wh u_2(x;s^\al)=0, & x\in\pa\Om,
\end{cases}\quad s^\al>s_2.
\]
In comparison with \eqref{equ-lap-v}, it follows from the uniqueness result for
boundary value problems of elliptic type that
\[
\wh u_2(x;s^\al)=\wh u_1(x;s)=s^{1-\al}\wh u(x;s),\quad
(x;s)\in\Om\times\{s>s_0\},\ s_0:=\max\{s_2^{1/\al},s_1\}.
\]
Since \eqref{eq-van} gives $\wh u(x;s)=0$ in $\om\times(0,\infty)$, we conclude
from the above identities that
\[
\wh u_2(x;\eta)=0,\quad(x;\eta)\in\om\times\{\eta>s_0^\al\}.
\]
Consequently, the uniqueness of the inverse Laplace transform indicates $u_2=0$
in $\om\times(0,\infty)$. According to the unique continuation property for
parabolic equations (see, e.g., \cite{SS87}), we conclude $u_2=0$ in
$\Om\times(0,\infty)$ and thus $a=u_2(\,\cdot\,,0)=0$ in $\Om$. Now that the
initial value vanishes, it is readily seen that $u=0$ in $\Om\times(0,\infty)$
from the uniqueness of the solution to \eqref{eq-ibvp-u}, which completes the
proof of the first part of Theorem \ref{thm-wuc}.\medskip

{\bf Case 2}\ \ $B\equiv0$, $c\ge0$ in $\Om$. Recall that in this case, we have
introduced the eigensystem $\{(\la_n,\vp_n)\}$ of the elliptic operator $\cA$.
According to the proof of \cite[Lemma 4.1]{LLY15}, the Laplace transform
$\wh u(\,\cdot\,;s)$ of the solution $u(\,\cdot\,,t)$ to \eqref{eq-ibvp-u}
reads
\[
\wh u(\,\cdot\,;s)=\f{h(s)}s\sum_{n=1}^\infty\f{(a,\vp_n)}{h(s)+\la_n}\vp_n,
\quad\rRe\,s>s_3,
\]
where $h(s):=\sum_{j=1}^mq_js^{\al_j}$ and $s_3>0$ is a sufficiently large
constant. Then it follows from \eqref{eq-van} that
\[
\f{h(s)}s\sum_{n=1}^\infty\f{(a,\vp_n)}{h(s)+\la_n}\vp_n=0
\quad\mbox{in }\om,\ \rRe\,s>s_3.
\]
Setting $\eta=h(s)$, we see that $\eta$ varies over some domain $U\subset\BC$
as $s$ varies over $\rRe\,s>s_3$. Therefore, we obtain
\begin{equation}\label{identity zero}
\sum_{n=1}^\infty\f{(a,\vp_n)}{\eta+\la_n}\vp_n=0
\quad\mbox{in }\om,\ \eta\in U.
\end{equation}
Moreover, it is readily seen that \eqref{identity zero} holds for
$\eta\in\BC\setminus\{-\la_n\}_{n=1}^\infty$. Then for any $n=1,2,\ldots$, we
can take a sufficiently small circle centered at $-\la_n$ which does not
include distinct eigenvalues, and integrating \eqref{identity zero} on this
circle yields
\[
u_n:=\sum_{\{k;\,\la_k=\la_n\}}(a,\vp_k)\vp_k=0
\quad\mbox{in }\om,\ \forall\,n=1,2,\ldots.
\]
Since $u_n$ satisfies the elliptic equation $(\cA-\la_n)u_n=0$ in $\Om$, the
unique continuation for elliptic equations implies $u_n=0$ in $\Om$ for all
$n=1,2,\ldots$. By the orthogonality of $\{\vp_n\}$ in $L^2(\Om)$, we conclude
$(a,\vp_n)=0$ for all $n=1,2,\ldots$ and thus $a=u(\,\cdot\,,0)=0$ in $\Om$,
which indicates $u=0$ in $\Om\times(0,\infty)$ again by the uniqueness of the
solution to \eqref{eq-ibvp-u}. This completes the proof of Theorem
\ref{thm-wuc}.
\end{proof}

Now let us turn to the proof of the uniqueness of the inverse source problem.
The argument is mainly based on the weak unique continuation and the following
Duhamel's principle for time-fractional parabolic equations.

\begin{lem}[Duhamel's principle]\label{lem-Duhamel}
Let $a=0$ and $F(x,t)=f(x)\,\mu(t)$, where $f\in L^2(\Om)$ and
$\mu\in C^1[0,T]$. Then the weak solution $u$ to the initial-boundary value
problem \eqref{eq-ibvp-u} allows the representation
\begin{equation}\label{eq-Duhamel}
u(\,\cdot\,,t)=\int_0^t\te(t-s)\,v(\,\cdot\,,s)\rd s,\quad0<t<T,
\end{equation}
where $v$ solves the homogeneous problem
\begin{equation}\label{equ-homo}
\begin{cases}
\cP v=0 & \mbox{in }Q,\\
v=f & \mbox{in }\Om\times\{0\},\\
v=0\mbox{ or }\pa_\cA v=0 & \mbox{on }\pa\Om\times(0,T)
\end{cases}
\end{equation}
and $\te\in L^1(0,T)$ is the unique solution to the fractional integral
equation
\begin{equation}\label{eq-FIE-te}
\sum_{j=1}^mq_jJ_{0+}^{1-\al_j}\te(t)=\mu(t),\quad0<t<T.
\end{equation}
\end{lem}

The above conclusion is almost identical to \cite[Lemma 4.1]{LRY16} for the
single-term case and \cite[Lemma 4.2]{L15} for the multi-term case, except for
the existence of non-symmetric part. Since the same argument still works in our
setting, we omit the proof here.

\begin{proof}[Proof of Theorem $\ref{thm-isp}$]
Let $u$ satisfy the initial-boundary value problem \eqref{eq-ibvp-u} with $a=0$
and $F(x,t)=f(x)\,\mu(t)$, where $f\in L^2(\Om)$ and $\mu\in C^1[0,T]$. Then
$u$ takes the form of \eqref{eq-Duhamel} according to Lemma \ref{lem-Duhamel}.
Performing the linear combination $\sum_{j=1}^mq_jJ_{0+}^{1-\al_j}$ of the
Riemann-Liouville integral operators to \eqref{eq-Duhamel}, we deduce
\begin{align*}
\sum_{j=1}^mq_jJ_{0+}^{1-\al_j}u(\,\cdot\,,t)
& =\sum_{j=1}^m\f{q_j}{\Ga(1-\al_j)}\int_0^t\f1{(t-\tau)^{\al_j}}
\int_0^\tau\te(\tau-\xi)\,v(\,\cdot\,,\xi)\,\rd\xi\rd\tau\\
& =\sum_{j=1}^m\f{q_j}{\Ga(1-\al_j)}\int_0^tv(\,\cdot\,,\xi)
\int_\xi^t\f{\te(\tau-\xi)}{(t-\tau)^{\al_j}}\,\rd\tau\rd\xi\\
& =\int_0^tv(\,\cdot\,,\xi)\sum_{j=1}^m\f{q_j}{\Ga(1-\al_j)}
\int_0^{t-\xi}\f{\te(\tau)}{(t-\xi-\tau)^{\al_j}}\,\rd\tau\rd\xi\\
& =\int_0^tv(\,\cdot\,,\xi)\sum_{j=1}^mq_jJ_{0+}^{1-\al_j}\te(t-\xi)\,\rd\xi
=\int_0^t\mu(t-\tau)\,v(\,\cdot\,,\tau)\,\rd\tau,
\end{align*}
where we applied Fubini's theorem and used the relation \eqref{eq-FIE-te}. Then
the vanishment of $u$ in $\om\times(0,T)$ immediately yields
\[
\int_0^t\mu(t-\tau)\,v(\,\cdot\,,\tau)\,\rd\tau=0\quad\mbox{in }\om,\ 0<t<T.
\]
Differentiating the above equality with respect to $t$, we obtain
\[
\mu(0)\,v(\,\cdot\,,t)+\int_0^t\mu'(t-\tau)\,v(\,\cdot\,,\tau)\rd\tau=0,
\quad\mbox{in }\om,\ 0<t<T.
\]
Owing to the assumption that $|\mu(0)|\ne0$, we estimate
\begin{align*}
\|v(\,\cdot\,,t)\|_{L^2(\om)} & \le\f1{|\mu(0)|}
\int_0^t|\mu'(t-\tau)|\|v(\,\cdot\,,\tau)\|_{L^2(\om)}\,\rd\tau\\
& \le\f{\|\mu\|_{C^1[0,T]}}{|\mu(0)|}
\int_0^t\|v(\,\cdot\,,\tau)\|_{L^2(\om)}\,\rd\tau,\quad0<t<T.
\end{align*}
Taking advantage of Gronwall's inequality, we conclude $v=0$ in
$\om\times(0,T)$. Finally, we apply Theorem \ref{thm-wuc} to the homogeneous
problem \eqref{equ-homo} to derive $v=0$ in $\Om\times(0,\infty)$, implying
$f=v(\,\cdot\,,0)=0$. This completes the proof of Theorem \ref{thm-isp}.
\end{proof}

\Section{Iterative Thresholding Algorithm}\label{sec-ita}

Based on the theoretical uniqueness result explained in the previous sections,
this section mainly aims at developing an effective numerical method for
Problem \ref{prob-isp}, that is, the numerical reconstruction of the spatial
component of the source term in a time-fractional parabolic equation.

As a representative, in the sequel we consider the initial-boundary value
problem for a single-term time-fractional diffusion equation with the
homogeneous Neumann boundary condition
\begin{equation}\label{equ-u(f)}
\begin{cases}
\pa_t^\al u(x,t)+\cA u(x,t)=f(x)\,\mu(t), & (x,t)\in Q,\\
u(x,0)=0, & x\in\Om,\\
\pa_\cA u(x,t)=0, & (x,t)\in\pa\Om\times(0,T).
\end{cases}
\end{equation}
For later use, we write the solution to \eqref{equ-u(f)} as $u(f)$ to emphasize
its dependency upon the unknown function $f$. From Lemma \ref{lem-fp-weak}, we
point out that $u(f)$ satisfies
\begin{equation}\label{eq-weak-u}
\int_Q\left(\sum_{i,j=1}^da_{ij}\,\pa_iu(f)\,\pa_jw+c\,u(f)\,w
+u(f)\,D_t^\al w\right)\rd x\rd t=\int_Qf\,\mu\,w\,\rd x\rd t
\end{equation}
for any test function $w\in H^\al(0,T;L^2(\Om))\cap L^2(0,T;H^1(\Om))$ with
$J^{1-\al}_{T-}w=0$ in $\Om\times\{T\}$, where $D_t^\al$ stands for the
backward Riemann-Liouville derivative. This is easily understood in view of
integration by parts and the following lemma.

\begin{lem}\label{lem-RL-ibp}
For $\al>0$ and $g_1,g_2\in L^2(0,T)$, there holds
\[
\int_0^T(J_{0+}^\al g_1(t))\,g_2(t)\,\rd t
=\int_0^Tg_1(t)\,J^\al_{T-}g_2(t)\,\rd t.
\]
\end{lem}

Henceforth, we specify $f_\true\in L^2(\Om)$ as the true solution to Problem
\ref{prob-isp} and investigate the numerical reconstruction by the noise
contaminated observation data $u^\de$ in $\om\times(0,T)$ satisfying
$\|u^\de-u(f_\true)\|_{L^2(\om\times(0,T))}\le\de$, where $\de$ stands for the
noise level. For avoiding ambiguity, we interpret $u^\de=0$ out of
$\om\times(0,T)$ so that it is well-defined in $Q$.

By a classical Tikhonov regularization technique, the reconstruction of the
source term can be reformulated as the minimization of the following output
least squares functional
\begin{equation}\label{equ-min}
\min_{f\in L^2(\Om)}\Phi(f),\quad
\Phi(f):=\|u(f)-u^\de\|_{L^2(\om\times(0,T))}^2+\rho\|f\|_{L^2(\Om)}^2,
\end{equation}
where $\rho>0$ is the regularization parameter. As the majority of efficient
iterative methods do, we need the information about the Fr\'echet derivative
$\Phi'(f)$ of the objective functional $\Phi(f)$. For an arbitrarily fixed
direction $g\in L^2(\Om)$, it follows from direct calculations that
\begin{align}
\Phi'(f)g & =2\int_0^T\!\!\!\int_\om\left(u(f)-u^\de\right)(u'(f)g)\,\rd x\rd t
+2\rho\int_\Om f\,g\,\rd x\rd t\nonumber\\
& =2\int_0^T\!\!\!\int_\om\left(u(f)-u^\de\right)u(g)\,\rd x\rd t
+2\rho\int_\Om f\,g\,\rd x\rd t.\label{equ-frechet}
\end{align}
Here $u'(f)g$ denotes the Fr\'echet derivative of $u(f)$ in the direction $g$,
and the linearity of \eqref{equ-u(f)} immediately yields
\[
u'(f)g=\lim_{\ep\to0}\f{u(f+\ep g)-u(f)}\ep=u(g).
\]
Obviously, it is extremely expensive to use \eqref{equ-frechet} to evaluate
$\Phi'(f)g$ for all $g\in L^2(\Om)$, since one should solve system
\eqref{equ-u(f)} for $u(g)$ with $g$ varying in $L^2(\Om)$ in the computation
for a fixed $f$.

In order to reduce the computational costs for computing the Fr\'echet
derivatives, we follow the argument used in \cite{LJY15} to introduce the
adjoint system of \eqref{equ-u(f)}, that is, the following system for a
backward time-fractional diffusion equation
\begin{equation}\label{equ-dual}
\begin{cases}
D_t^\al z+\cA z=F & \mbox{in }Q,\\
J^{1-\al}_{T-}z=0 & \mbox{in }\Om\times\{T\},\\
\pa_\cA z=0 & \mbox{on }\pa\Om\times(0,T).
\end{cases}
\end{equation}
Parallelly to Definition \ref{def-fp-weak}, we give the definition of the weak
solution to the backward fractional diffusion equation with Riemann-Liouville
derivatives.

\begin{defi}\label{def-bp-weak}
Let $F\in L^2(Q)$. We say that $z$ is a weak solution to \eqref{equ-dual} if
\begin{alignat*}{2}
& z\in L^2(0,T;\cD(\wt\cA\,)), & \quad & D_t^\al z+\cA z=F\mbox{ in }L^2(Q),\\
& J_{T-}^{1-\al}z\in C([0,T];L^2(\Om)), & \quad
& \lim_{t\to T}\|J_{T-}^{1-\al}z(\,\cdot\,,t)\|_{L^2(\Om)}=0.
\end{alignat*}
\end{defi}

Correspondingly, we can also show the well-posedness of the solution defined
above as that in Lemma \ref{lem-fp-weak}.

\begin{lem}[Well-posedness for Definition \ref{def-bp-weak}]\label{lem-bp-weak}
Let $F\in L^2(Q)$. Then the problem \eqref{equ-dual} admits a unique weak
solution $z\in L^2(0,T;\cD(\wt\cA\,))$ such that $D_t^\al z\in L^2(Q)$.
Moreover, there exists a constant $C>0$ such that
\[
\|D_t^\al z\|_{L^2(Q)}+\|z\|_{L^2(0,T;\cD(\wt\cA\,))}\le C\|F\|_{L^2(Q)}.
\]
\end{lem}

In a similar manner of the proof of \cite[Proposition 4.1]{F14}, one can also
show Lemma \ref{lem-bp-weak} by using the eigenfunction expansion. For
conciseness, we omit the proof here. On the other hand, from Lemma
\ref{lem-bp-weak} and integration by parts, it turns out that the solution $z$
to problem \eqref{equ-dual} satisfies
\begin{equation}\label{eq-weak-z}
\int_Q\left(\sum_{i,j=1}^da_{ij}\,\pa_iz\,\pa_jw +c\,z\,w+(D_t^\al z)\,w\right)
\rd x\rd t=\int_QF\,w\,\rd x\rd t
\end{equation}
for any test function $w\in L^2(0,T;H^1(\Om))$ with $w=0$ in $\Om\times\{0\}$.

Based on the above argument, we now introduce the adjoint system of
\eqref{equ-u(f)} associated with Problem \ref{prob-isp} as
\begin{equation}\label{equ-v(f)}
\begin{cases}
D_t^\al z+\cA z=\chi_\om\left(u(f)-u^\de\right) & \mbox{in }Q,\\
J^{1-\al}_{T-}z=0 & \mbox{in }\Om\times\{T\},\\
\pa_\cA z=0 & \mbox{on }\pa\Om\times(0,T).
\end{cases}
\end{equation}
Here $\chi_\om$ denotes the characterization function of $\om$, and we write
the solution of \eqref{equ-v(f)} as $z(f)$. Then for any $f,g\in L^2(\Om)$, it
follows from Lemma \ref{lem-RL-ibp} and Remark \ref{rem-analy} that $z(f)$ and
$u(g)$ can be taken as mutual test functions in definitions \eqref{eq-weak-u}
and \eqref{eq-weak-z}. In such a manner, we can further treat the first term in
\eqref{equ-frechet} as
\begin{align*}
& \quad\,\int_0^T\!\!\!\int_\om\left(u(f)-u^\de\right)u(g)\,\rd x\rd t
=\int_Q\chi_\om\left(u(f)-u^\de\right)u(g)\,\rd x\rd t\\
& =\int_Q\left(\sum_{i,j=1}^da_{ij}\,\pa_iz(f)\,\pa_ju(g)+c\,z(f)\,u(g)
+(D_t^\al z(f))\,u(g)\right)\rd x\rd t=\int_Qg\,\mu\,z(f)\,\rd x\rd t,
\end{align*}
implying
\[
\Phi'(f)g=2\int_\Om\left(\int_0^T\mu\,z(f)\,\rd t+\rho\,f\right)g\,\rd x,
\quad\forall\,g\in L^2(\Om).
\]
This suggests a characterization of the solution to the minimization problem
\eqref{equ-min}.

\begin{lem}
$f^*\in L^2(\Om)$ is a minimizer of the functional $\Phi(f)$ in \eqref{equ-min}
only if it satisfies the variational equation
\begin{equation}\label{j1}
\int_0^T\mu\,z(f^*)\,\rd t+\rho\,f^*=0,
\end{equation}
where $z(f^*)$ solves the backward problem \eqref{equ-v(f)} with the
coefficient $f^*$.
\end{lem}

Adding $Mf^*$ ($M>0$) to both sides of \eqref{j1} and rearranging in view of
the iteration, we are led to the iterative thresholding algorithm
\begin{equation}\label{j2}
f_{k+1}=\f M{M+\rho}f_k-\f1{M+\rho}\int_0^T\mu\,z(f_k)\,\rd t,
\quad k=0,1,\ldots,
\end{equation}
where $M>0$ is a tuning parameter for the convergence. Similarly to
\cite{LJY15}, it follows from the general theory stated in \cite{DDM04} that it
suffices to choose
\begin{equation}\label{eq-cov}
M\ge\|A\|_{\mathrm{op}}^2,\quad\mbox{where}\quad\begin{aligned}
A:L^2(\Om) & \to L^2(\om\times(0,T)),\\
f & \mapsto u(f)|_{\om\times(0,T)}.
\end{aligned}
\end{equation}
At this stage, we are well prepared to propose the iterative thresholding
algorithm for the reconstruction.

\begin{algo}\label{algori-itera}
Choose a tolerance $\ve>0$, a regularization parameter $\rho>0$ and a tuning
constant $M>0$ according to \eqref{eq-cov}. Give an initial guess
$f_0\in L^2(\Om)$, and set $k=0$.
\begin{enumerate}
\item Compute $f_{k+1}$ by the iterative update \eqref{j2}.
\item If $\|f_{k+1}-f_k\|_{L^2(\Om)}/\|f_k\|_{L^2(\Om)}<\ve$, stop the
iteration. Otherwise, update $k\leftarrow k+1$ and return to Step 1.
\end{enumerate}
\end{algo}

By \cite[Theorem 3.1]{DDM04}, we see that the sequence $\{f_k\}_{k=1}^\infty$
generated by the iteration \eqref{j2} converges strongly to the solution of the
minimization problem \eqref{equ-min}. Meanwhile, we can also see from
\eqref{j2} that at each iteration step, we only need to solve the forward
problem \eqref{equ-u(f)} once for $u(f_k)$ and the backward problem
\eqref{equ-v(f)} once for $z(f^k)$ subsequently. As a result, the numerical
implementation of Algorithm \ref{algori-itera} is easy and computationally
cheap. Moreover, although \eqref{equ-v(f)} involves the backward
Riemann-Liouville derivative, we know that the solution $z(f)$ coincides with
the following problem with a backward Caputo derivative
\begin{equation}\label{eq-gov-z}
\begin{cases}
-J_{T-}^{1-\al}(\pa_tz)+\cA z=\chi_\om\left(u(f)-u^\de\right) & \mbox{in }Q,\\
z=0 & \mbox{in }\Om\times\{0\},\\
\pa_\cA z=0 & \mbox{on }\pa\Om\times(0,T),
\end{cases}
\end{equation}
thanks to the homogeneous terminal value $J_{T-}^{1-\al}z(\,\cdot\,,T)=0$.
Therefore, in the numerical simulation it suffices to deal with
\eqref{eq-gov-z} instead of \eqref{equ-v(f)} by the same forward solver for
\eqref{equ-u(f)}.

\Section{Numerical Experiments}\label{sec-numer}

In this section, we will apply the iterative thresholding algorithm established
in the previous section to the numerical treatment of Problem \ref{prob-isp} in
one and two spatial dimensions, that is, the identification of the spatial
component $f$ in the source term of the initial-boundary value problem
\eqref{equ-u(f)}.

To begin with, we assign the general settings of the reconstructions as
follows. Without loss of generality, in \eqref{equ-u(f)} we set
\[
\Om=(0,1)^d\ (d=1,2),\quad T=1,\quad\cA u=-\triangle u+u.
\]
With the true solution $f_\true\in L^2(\Om)$, we produce the noisy observation
data $u^\de$ by adding uniform random noises to the true data, i.e.,
\[
u^\de(x,t)=(1+\de\,\rand(-1,1))\,u(f_\true)(x,t),\quad(x,t)\in\om\times(0,T).
\]
Here $\rand(-1,1)$ denotes the uniformly distributed random number in $[-1,1]$
and $\de\ge0$ is the noise level. Throughout this section, we will fix the
known temporal component $\mu$ in the source term, the regularization parameter
$\rho$ and the initial guess $f_0$ as
\[
\mu(t)=1+10\pi\,t^2,\quad\rho=10^{-5},\quad f_0(x)\equiv2
\]
respectively. We shall demonstrate the reconstruction method by abundant test
examples in one and two spatial dimensions. Other than the illustrative
figures, we mainly evaluate the numerical performance by the relative
$L^2$-norm error
\[
\err:=\f{\|f_K-f_\true\|_{L^2(\Om)}}{\|f_\true\|_{L^2(\Om)}}
\]
and the number $K$ of iterations, where $f_K$ is regarded as the reconstructed
solution produced by Algorithm \ref{algori-itera}. The forward problem
\eqref{equ-u(f)} and the backward problem \eqref{eq-gov-z} involved in
Algorithm \ref{algori-itera} are solved by the numerical scheme proposed in
\cite{Lin07}, which is composed of a finite difference method in time and the
Legendre spectral method in space.

We start from the one-dimensional case. We divide the space-time region
$[0,1]\times[0,1]$ into a $40\times40$ equidistant mesh, and set the tolerance
parameter $\ve=10^{-3}$ in Algorithm \ref{algori-itera}. Except for the factors
mentioned above, we will test the numerical performance of the proposed
algorithm with different choices of true solution $f_\true$, fractional order
$\al$, noise level $\de$ and observation subdomain $\om$.

\begin{ex}\label{ex1.1}
First we fix the noise level $\de=2\%$ and the observation subdomain
$\om=(0,0.05)\cup(0.95,1)$ and test the algorithm with the following settings:

(a) $\al=0.3$, $f_\true(x)=\sin(\pi x)+x-3$, $M=2$.

(b) $\al=0.5$, $f_\true(x)=\sin(\pi x)-3/2$, $M=1$.

\noindent In Figure \ref{a1} we illustrate the comparisons of recovered
solutions with the true ones, and show the iteration steps $K$ and the relative
error $\err$.
\begin{figure}[htbp]\centering
\includegraphics[trim=1mm 2mm 12mm 6mm,clip=true,width=.45\textwidth]{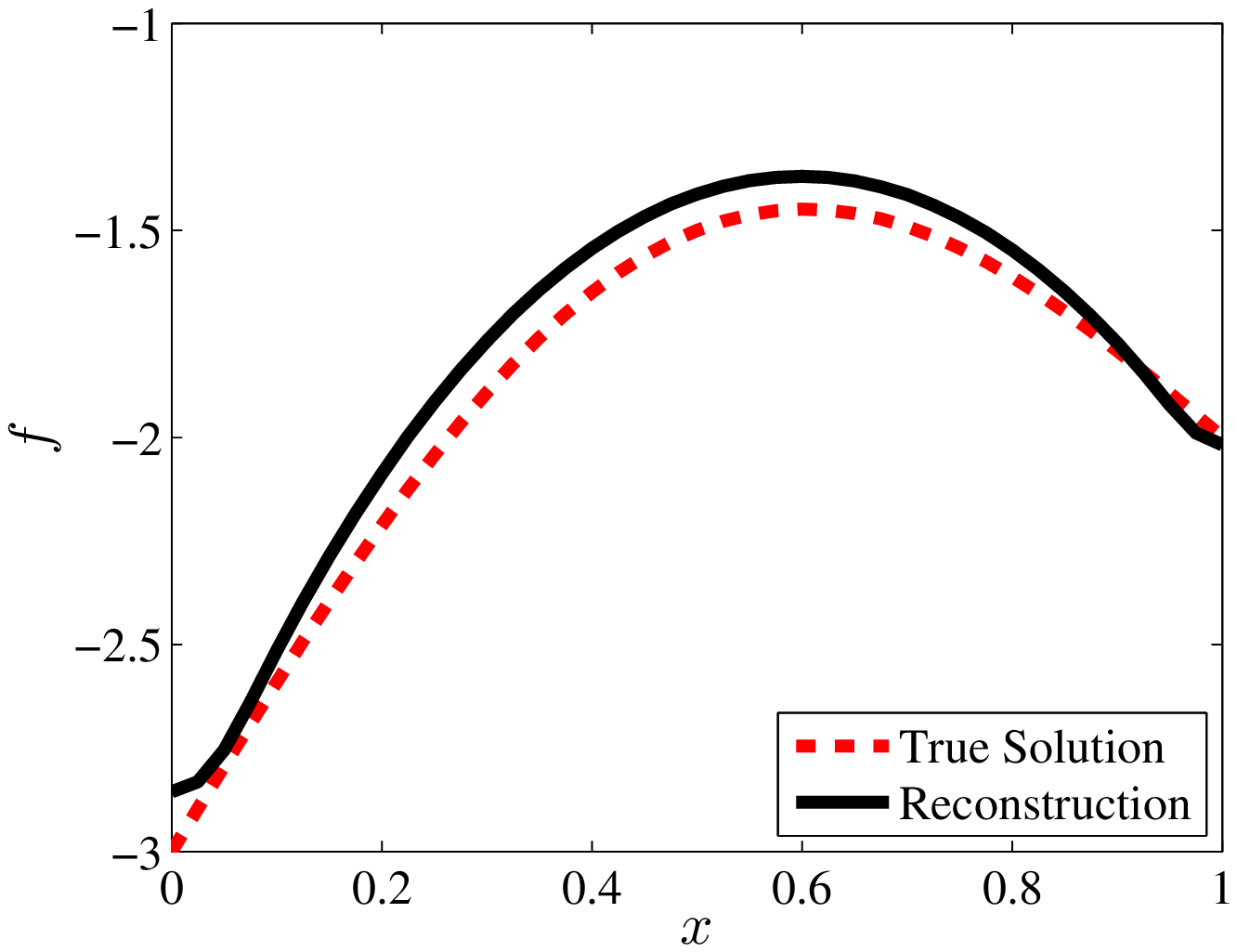}
\qquad
\includegraphics[trim=1mm 2mm 12mm 6mm,clip=true,width=.45\textwidth]{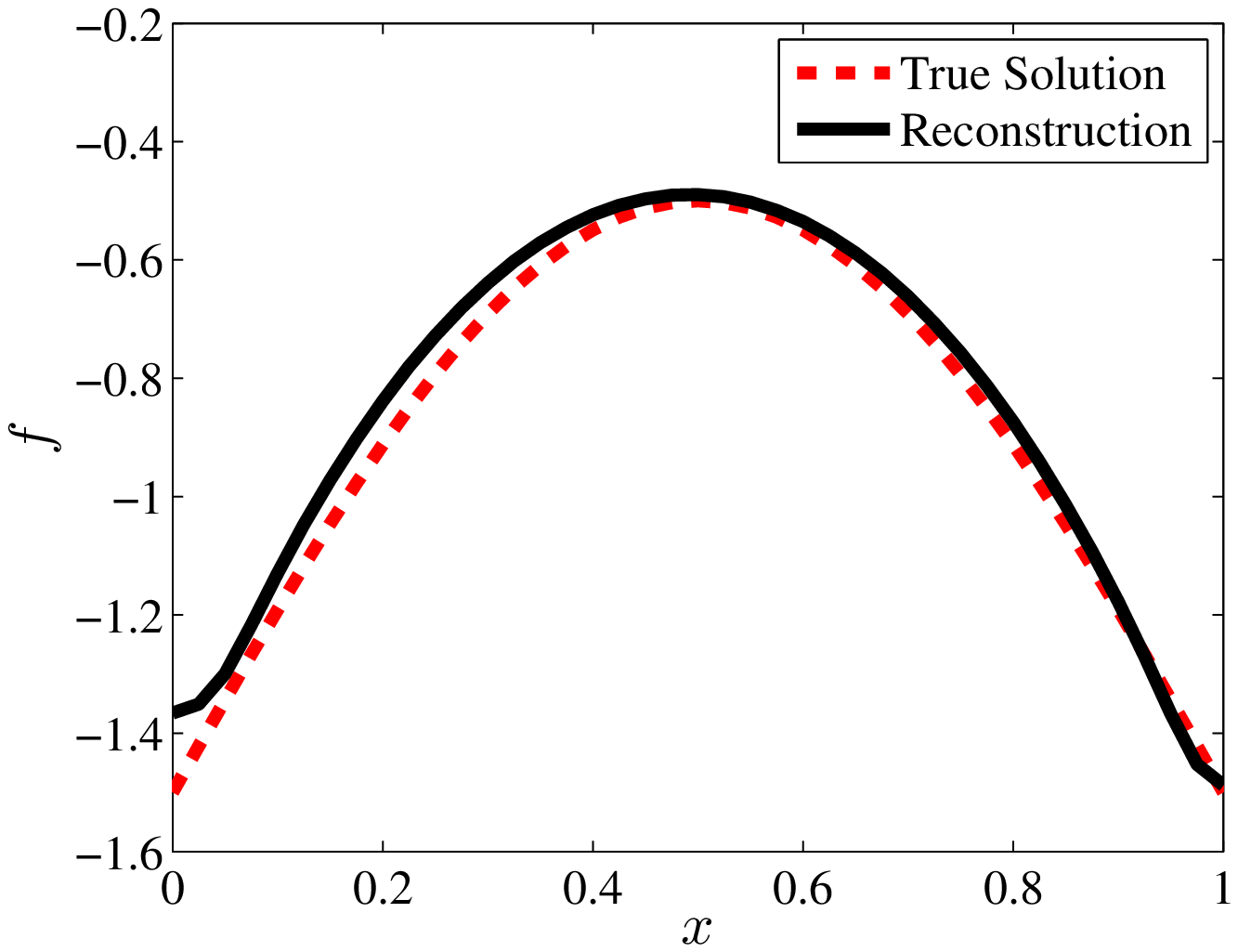}\\
\caption{True solutions $f_\true$ and their reconstructions $f_K$ obtained in
Example \ref{ex1.1}. Left: Case (a), $K=16$, $\err=4.56\%$;
Right: Case (b), $K=49$,  $\err=4.92\%$.}\label{a1}
\end{figure}
\end{ex}\vspace{-5mm}

\begin{ex}\label{ex1.5}
In this example, we fix $\al=0.8$, $M=1$ and the true solution
$f_\true(x)=-\sin(\pi x/2)-x^2+3$. Our aim is to test the numerical performance
of Algorithm \ref{algori-itera} with various choices of the noise level $\de$
and the observation subdomain $\om$ to see their influences on the
reconstructions. First we fix $\om=(0,0.05)\cup(0.95,1)$ and enlarge $\de$ from
$0.5\%$, $1\%$, $2\%$ to $4\%$. Next, we fix the noise level as $\de=2\%$ and
shrink $\om$ from $(0,0.2)\cup(0.8,1)$, $(0,0.1)\cup(0.9,1)$ to
$(0,0.025)\cup(0.975,1)$. The choices of $\de$, $\om$ in the tests and the
corresponding numerical performances are listed in Table \ref{diff}.
\begin{table}[htbp]\centering
\caption{Choices of noise levels $\de$ and observation subdomains $\om$ along
with the corresponding iteration steps $K$ and the relative errors $\err$ in
Example \ref{ex1.5}.}\label{diff}
\begin{tabular}{cc|cc}
\hline\hline
$\de$ & $\om$ & $\err$ & $K$\\
\hline
$0.5\%$ & $(0,0.05)\cup(0.95,1)$ & $2.87\%$ & $51$\\
$1\%$ & $(0,0.05)\cup(0.95,1)$ & $3.61\%$ & $51$\\
$2\%$ & $(0,0.05)\cup(0.95,1)$ & $5.38\%$ & $51$\\
$4\%$ & $(0,0.05)\cup(0.95,1)$ & $9.35\%$ & $50$\\
\hline
$2\%$ & $(0,0.2)\cup(0.8,1)$ & $4.11\%$ & $20$\\
$2\%$ & $(0,0.1)\cup(0.9,1)$ & $4.05\%$ & $31$\\
$2\%$ & $(0,0.025)\cup(0.975,1)$ & $9.89\%$ & $79$\\
\hline\hline
\end{tabular}
\end{table}
\end{ex}

We can see from Figures \ref{a1} that with different fractional orders $\al$
and a $2\%$ noise in the data, the numerical reconstruction $f_K$ appear to be
quite satisfactory in view of the highly ill-posedness of the inverse source
problem, even with very bad initial constant guesses and very small sizes of
the observation subdomain $\om$. What's more, we can observe from Table
\ref{diff} that  Algorithm \ref{algori-itera} have two important advantages,
namely, it processes strong robustness against the oscillating measurement
errors, and it is not sensitive to the smallness of the observation subdomain
$\om$.

Now we proceed to the more challenging two-dimensional case, where we divide
the space-time region $\ov\Om\times[0,T]=[0,1]^2\times[0,1]$ into a
$40^2\times40$ equidistant mesh. Similarly to the one-dimensional examples, we
will test the numerical performance of Algorithm \ref{algori-itera} from
various aspects, including different choices of true solutions, noise levels
and observation subdomains. For simplicity, we unify the tuning parameter in
Algorithm \ref{algori-itera} as $M=2$ in the following examples.

\begin{ex}\label{ex2.1}
Fix the noise level as $\de=1\%$. We choose the observation subdomain and the
tolerance parameter as $\om=\Om\setminus[0.1,0.9]^2$ and $\ve=\de/3$,
respectively. We specify two pairs of fractional orders and true solutions as
follows.

(a) $\al=0.3$, $f_\true(x)=f_\true(x_1,x_2)=x_1+x_2+1$.

(b) $\al=0.5$, $f_\true(x)=\cos(\pi x_1)\cos(\pi x_2)+2$.

\noindent Parallelly to Example \ref{ex1.1}, we compare the recovered solutions
with the true ones, and show the iteration steps $K$ and the relative error
$\err$ in Figure \ref{b1}.
\begin{figure}[htbp]\centering
\includegraphics[trim=1mm 5mm 9mm 8mm,clip=true,width=.45\textwidth]{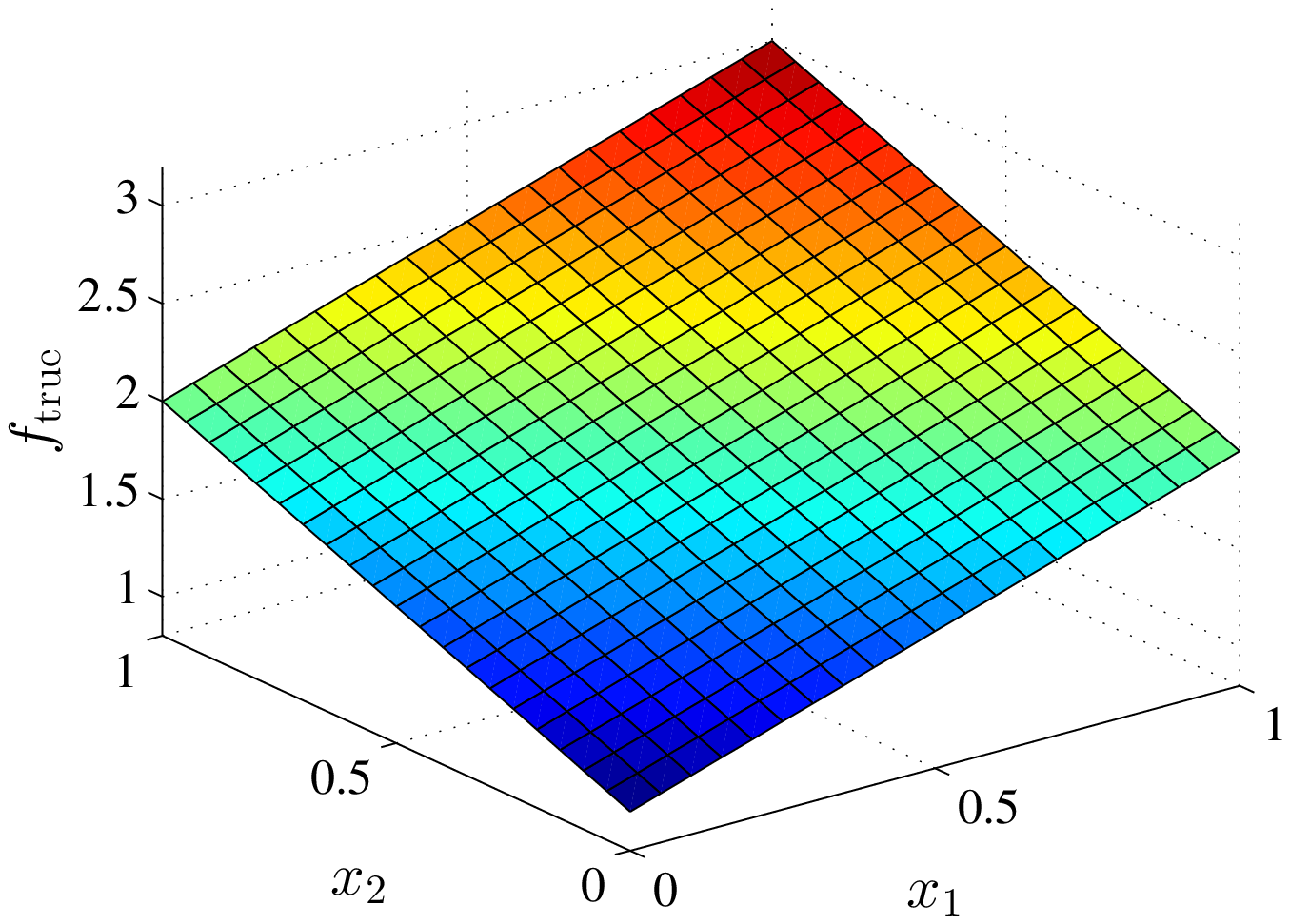}
\qquad
\includegraphics[trim=1mm 5mm 9mm 8mm,clip=true,width=.45\textwidth]{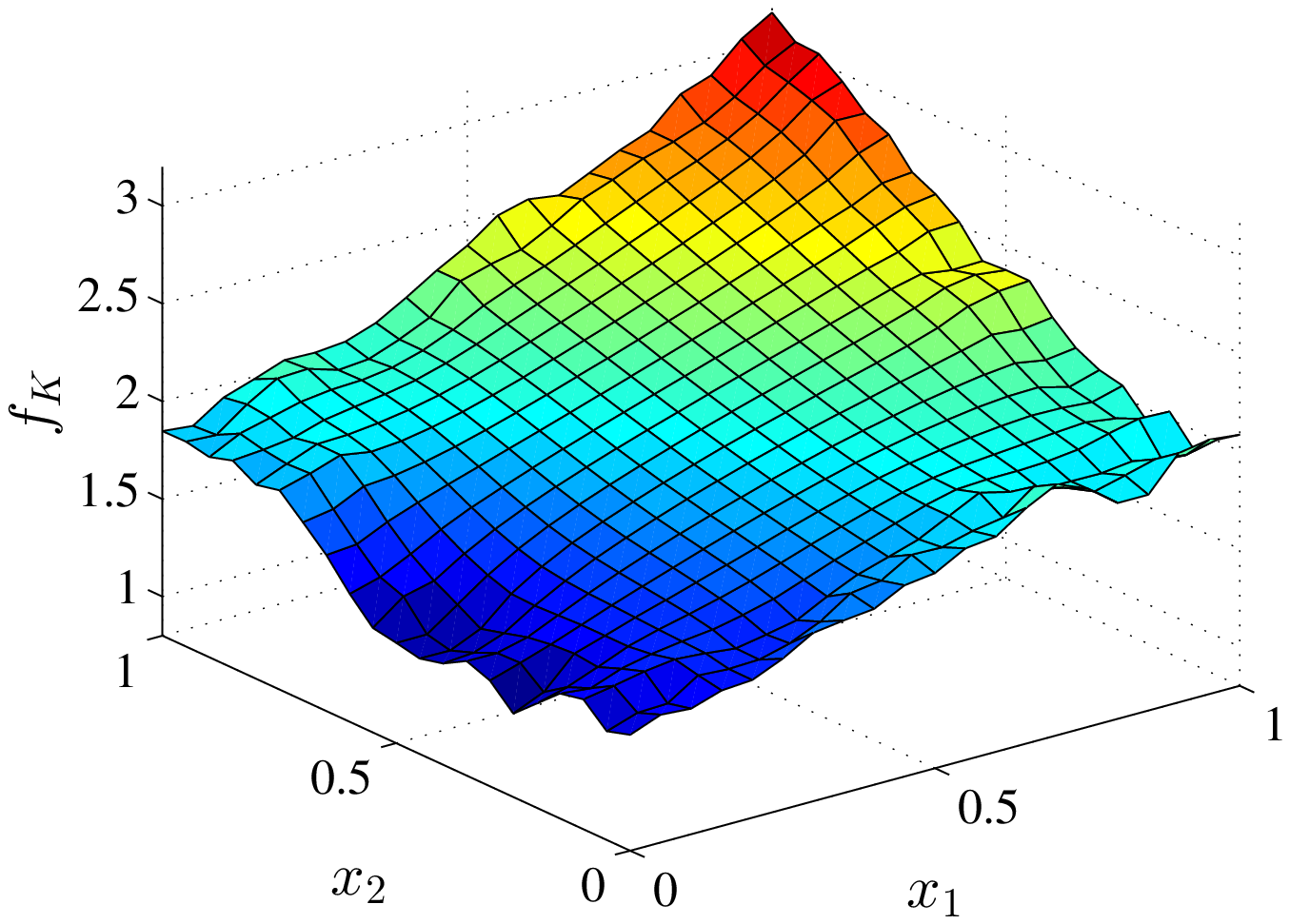}\\[5mm]
\includegraphics[trim=1mm 5mm 9mm 8mm,clip=true,width=.45\textwidth]{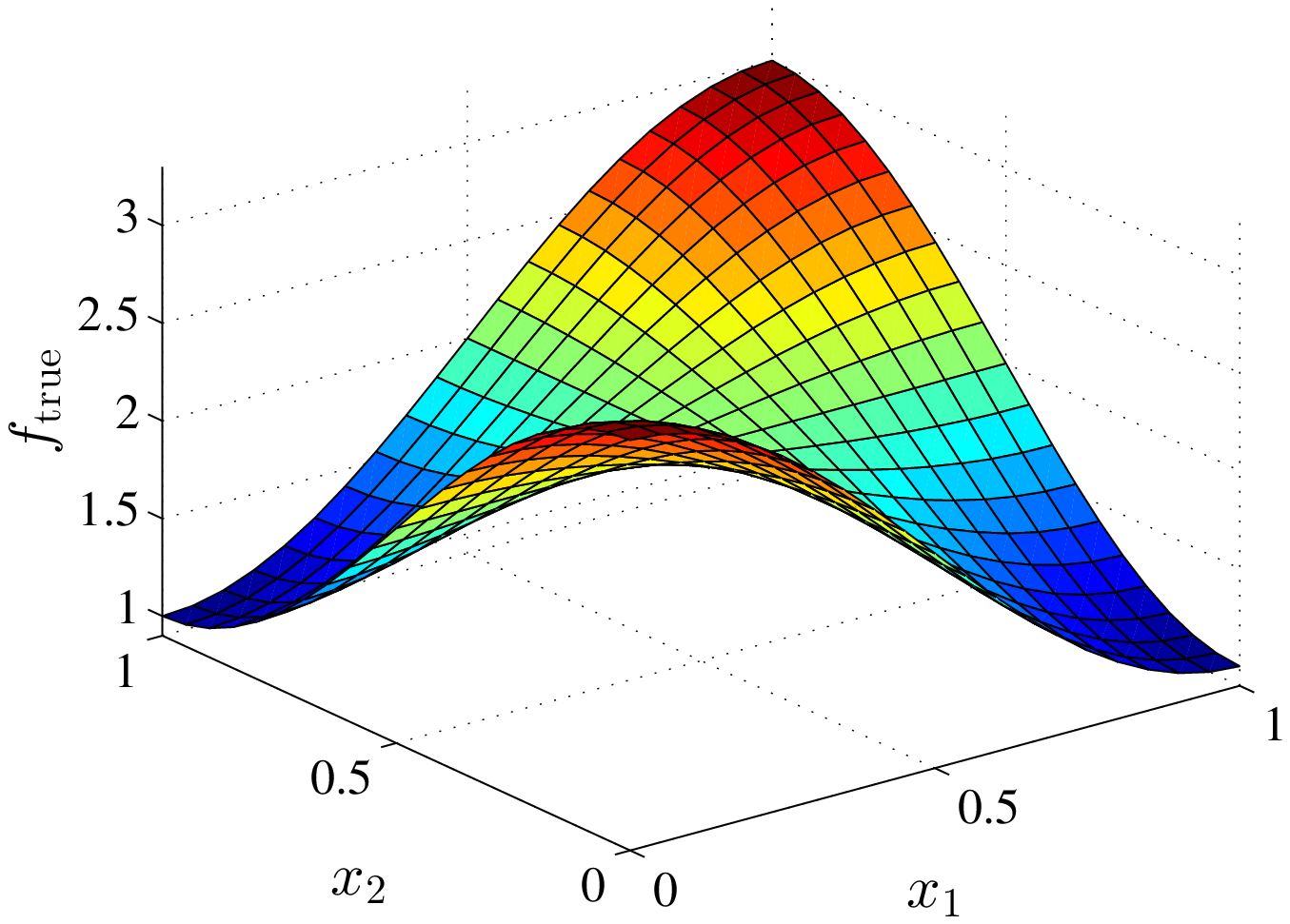}
\qquad
\includegraphics[trim=1mm 5mm 9mm 8mm,clip=true,width=.45\textwidth]{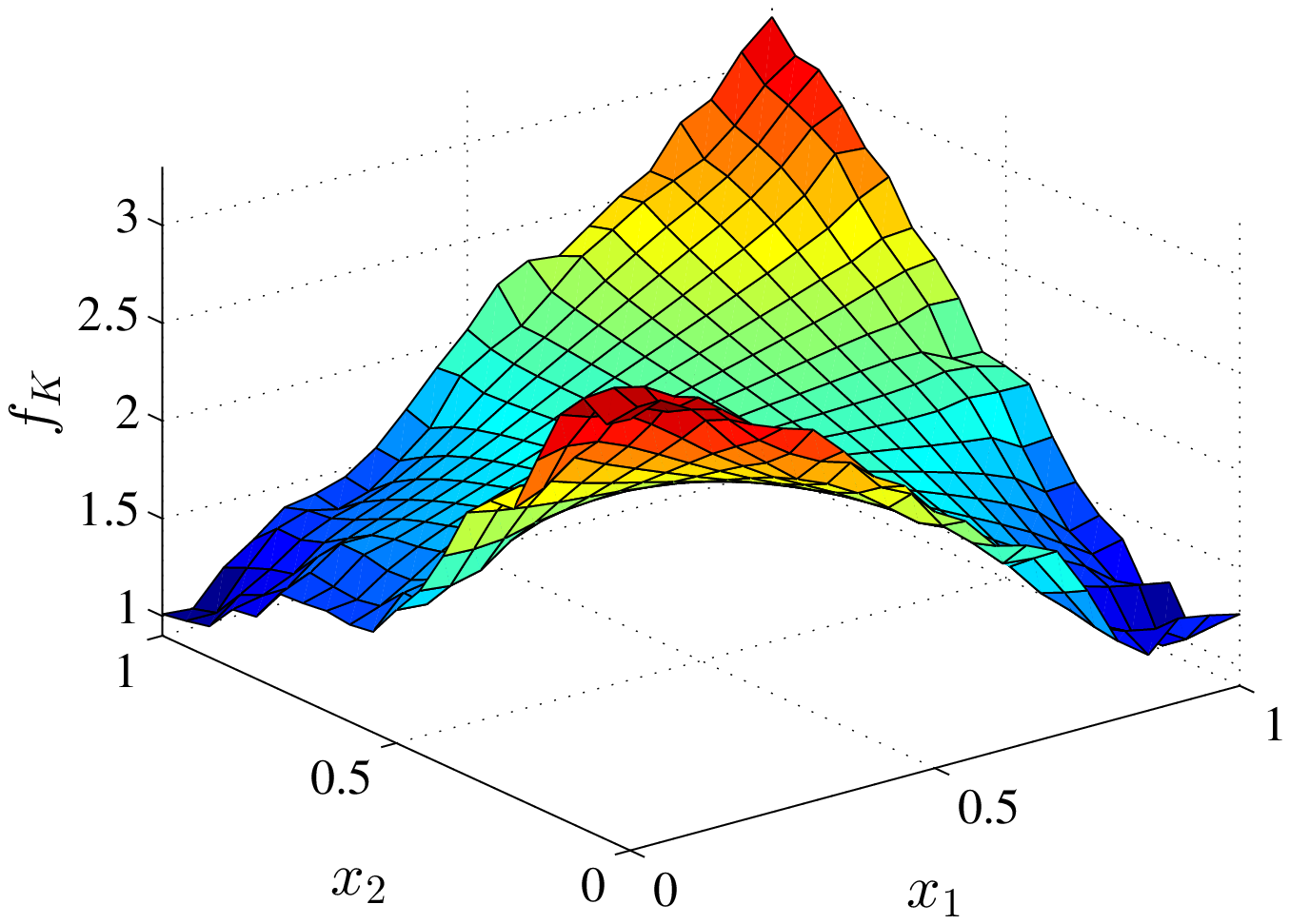}\\
\caption{True solutions $f_\true$ (left) and their reconstructions $f_K$
(right) obtained in Example \ref{ex2.1}. Top: Case (a), $K=21$, $\err=6.21\%$;
Bottom: Case (b), $K=36$, $\err=7.17\%$.}\label{b1}
\end{figure}
\end{ex}

\begin{ex}\label{ex2.3}
The aim of this example is the same as that of Example \ref{ex1.5}, that is, to
see the behavior of the reconstructions with respect to various choices of
noise levels $\de$ and observation subdomains $\om$. To this end, we fix the
fractional order $\al=0.8$ and the true solution
$f_\true(x)=\exp((x_1+x_2)/2)+1$, and choose the tolerance parameter as
$\ve=\de/5$. First, we fix $\om=\Om\setminus[0.1,0.9]^2$ as that in the
previous example, and change the noise levels as $\de=0.5\%$, $1\%$, $2\%$ and
$4\%$. Next, we fix $\de=1\%$ and take $\om$ as $\Om\setminus[0.1,0.8]^2$,
$\Om\setminus[0.05,0.95]^2$ and $\Om\setminus[0,0.9]\times[0.1,0.9]$.
Especially, we see that in the last choice, $\om$ only covers three edges of
$\pa\Om$. We list the choices of $\de$, $\om$ in the tests and the
corresponding numerical performances in Table \ref{diff2}.
\begin{table}[htbp]\centering
\caption{Choices of noise levels $\de$ and observation subdomains $\om$ along
with the corresponding iteration steps $K$ and the relative errors $\err$ in
Example \ref{ex2.3}.}\label{diff2}
\centering\begin{tabular}{cc|cc}
\hline\hline
$\de$ & $\om$ & $\err$ & $K$\\
\hline
$0.5\%$ & $\Om\setminus[0.1,0.9]^2$ & $3.25\%$ & $35$\\
$1\%$ & $\Om\setminus[0.1,0.9]^2$ & $4.69\%$ & $26$\\
$2\%$ & $\Om\setminus[0.1,0.9]^2$ & $7.11\%$ & $17$\\
$4\%$ & $\Om\setminus[0.1,0.9]^2$ & $10.31\%$ & $8$\\
\hline
$1\%$ & $\Om\setminus[0.1,0.8]^2$ & $3.63\%$ & $21$\\
$1\%$ & $\Om\setminus[0.05,0.95]^2$ & $6.70\%$ & $42$\\
$1\%$ & $\Om\setminus[0,0.9]\times[0.1,0.9]$ & $5.46\%$ & $22$\\
\hline\hline
\end{tabular}
\end{table}
\end{ex}

It can be readily seen from the above two-dimensional examples that the
iterative thresholding algorithm shows almost the same numerical performances
as that in the one-dimensional case. As expected, the proposed algorithm
demonstrates a strong robustness against oscillating noises in the observation
data and a certain insensitivity to the smallness of the observation subdomain.
Nevertheless, we point out that the reconstructions here are not as accurate as
that in \cite{LJY15}, where a similar iterative method was applied to an
inverse source problem for hyperbolic-type equations. The reason most probably
roots in the underlying ill-posedness of Problem \ref{prob-isp} for fractional
parabolic equations, which is severer than that for hyperbolic ones.

\Section{Concluding Remarks}

In Theorem \ref{thm-isp}, we only proved the uniqueness result for the inverse
source problem. In comparison, it is known that conditional stability results
hold for the same inverse problems for parabolic or hyperbolic equations based
on Carleman estimates or the multiplier method. Unfortunately, such techniques
do not work in the case of fractional diffusion equations due to the absence of
the fundamental integration by parts for the fractional derivatives. This is
also a direct reason why the unique continuation was only established in the
weak sense (see Theorem \ref{thm-wuc}, Cheng et al. \cite{CLN13}, Lin and
Nakamura \cite{LN16}) and Xu et al. \cite{XCY11}).

In the numerical aspect, we reformulate Problem \ref{prob-isp} as a
minimization problem in the typical situation in the case of $B\equiv0$ and
$m=1$. Then we characterize the minimizer by a variational with the help of the
corresponding adjoint problem of \eqref{eq-ibvp-u}, which results in the
desired iterative thresholding algorithm. Then several numerical experiments
for the reconstructions are implemented to show the efficiency and accuracy of
the proposed Algorithm \ref{algori-itera}. Here we point out that in case of
the homogeneous Neumann boundary condition, it is necessary to assume $B=0$ in
Algorithm \ref{algori-itera}, since the adjoint system used to derive our
algorithm heavily relies on the symmetry of problem \eqref{equ-u(f)}. The
algorithm for the non-symmetric case remains open.

\appendix
\Section{Proof of Lemma \ref{lem-fp-weak}}\label{sec-app}

In this appendix, we provide the proof of Lemma \ref{lem-fp-weak}, namely, the
well-posedness of the weak solution to the inhomogeneous problem
\eqref{eq-ibvp-u} in the sense of fractional Sobolev spaces in time. To this
end, we introduce the usual Mittag-Leffler function (see, e.g.,
\cite[\S1.2.1]{P99})
\[
E_{\al,\be}(\ze):=\sum_{k=1}^\infty\f{\ze^k}{\Ga(\al k+\be)},
\quad\ze\in\BC,\ \al>0,\ \be\in\BR,
\]
by which we define a collection of solution operators $\{S_\al(t)\}_{t>0}$ as
\begin{equation}\label{def-S(t)}
\begin{aligned}
S_\al(t):L^2(\Om) & \to L^2(\Om),\\
\psi & \mapsto-t^{\al-1}\sum_{n=1}^\infty
E_{\al,\al}(-\la_nt^\al)(\psi,\vp_n)\vp_n.
\end{aligned}
\end{equation}
Moreover, it follows from \cite[Theorem 1.6]{P99} that there exists a constant
$C>0$ such that
\begin{equation} \label{esti-S(t)}
\|S_\al(t)\|_{\mathrm{op}}\le C\,t^{\al-1},\quad t>0.
\end{equation}

We are in a position to give the proof of Lemma \ref{lem-fp-weak}.

\begin{proof}[Proof of Lemma $\ref{lem-fp-weak}$]
Let $a=0$ and $F\in L^2(Q)$. Without loss of generality, we only treat the
multi-term case, i.e., $m\ge2$. Henceforth, $C>0$ denotes generic constants
which may change from line to line.

Regarding the terms of lower fractional orders and advection as the new source
terms, we can argue similarly as that in the proof in \cite{LY13} to see that
the solution formally satisfies the integral equation
\[
u(\,\cdot\,,t)=(\cK-\cL)u(\,\cdot\,,t)-\Psi(\,\cdot\,,t),
\]
where
\begin{align*}
& \cK:=\sum_{j=2}^mq_j\cK_j,\quad\cK_ju(\,\cdot\,,t):=\int_0^t
S_{\al_1}(t-\tau)\pa_\tau^{\al_j}u(\,\cdot\,,\tau)\,\rd\tau\ (j=2,\ldots,m),\\
& \cL u(\,\cdot\,,t):=\int_0^tS_{\al_1}(t-\tau)B\cdot\nb u(\,\cdot\,,\tau)\,
\rd\tau,\quad\Psi(\,\cdot\,,t):=\int_0^tS_{\al_1}(t-\tau)F(\,\cdot\,,\tau)\,
\rd\tau.\label{def-g}
\end{align*}
In the sequel, for $\eta\in(0,T]$ we define the space $X_\eta$ and its norm
$\|\cdot\|_{X_\eta}$ as
\[
X_\eta:=H^{\al_1}(0,\eta;L^2(\Om))\cap L^2(0,\eta;H^2(\Om)),\quad
\|\cdot\|_{X_\eta}:=\|\cdot\|_{H^{\al_1}(0,\eta;L^2(\Om))}
+\|\cdot\|_{L^2(0,\eta;H^2(\Om))},
\]
respectively. Recalling the operator $\wt\cA$ introduced in Section
\ref{sec-pre}, we have $\wt\cA\Psi=\cA\Psi+\Psi$, where
\[
\cA\Psi(\,\cdot\,,t)=\int_0^t(t-\tau)^{\al_1-1}\sum_{n=1}^\infty\la_n
E_{\al_1,\al_1}(-\la_n(t-\tau)^{\al_1})(F(\,\cdot\,,\tau),\vp_n)\vp_n\,\rd s
\]
by definition \eqref{def-S(t)}. First it follows immediately from
\eqref{esti-S(t)} and Young's inequality that
\[
\|\Psi\|_{L^2(\Om\times(0,\eta))}\le C\|F\|_{L^2(\Om\times(0,\eta))}.
\]
To estimate $\cA\Psi$, we take advantage of the basic properties of
Mittag-Leffler functions (see, e.g., \cite[Lemma 3.3]{SY11a}) to deduce
\begin{align*}
\la_n\int_0^\eta t^{\al_1-1}|E_{\al_1,\al_1}(-\la_nt^{\al_1})|\,\rd t
& =\la_n\int_0^\eta t^{\al_1-1}E_{\al_1,\al_1}(-\la_nt^{\al_1})\,\rd t\\
& =-\int_0^\eta\f\rd{\rd t}E_{\al_1,1}(-\la_nt^{\al_1})\,\rd t
=1-E_{\al_1,1}(-\la_n\eta^{\al_1}).
\end{align*}
From the boundedness of $E_{\al_1,1}(-\la_n\eta^{\al_1})$ and Young's
inequality, we obtain
\[
\|\cA\Psi\|_{L^2(\Om\times(0,\eta))}\le\sum_{n=1}^\infty
\left(\int_0^\eta|(F(\,\cdot\,,t),\vp_n)|^2\,\rd t\right)^{1/2}
(1-E_{\al_1,1}(-\la_n\eta^{\al_1}))\le C\|F\|_{L^2(\Om\times(0,\eta))},
\]
indicating
\[
\|\Psi\|_{L^2(0,\eta;H^2(\Om))}\le C\|F\|_{L^2(\Om\times(0,\eta))}.
\]
Now by an argument similar to the proof of \cite[Theorem 4.1]{GLY15}, we obtain
\[
\|\Psi\|_{H^{\al_1}(0,\eta;L^2(\Om))}\le C\|F\|_{L^2(\Om\times(0,\eta))}.
\]

Next we proceed to show that $\cK-\cL:X_\eta\to X_\eta$ is compact. In fact,
according to \cite[Theorem 4.2]{GLY15}, we have
\begin{equation}\label{esti-Lu}
\|\cL w\|_{X_\eta}\le C\|B\cdot\nb w\|_{L^2(\Om\times(0,\eta))}
\le C\|w\|_{L^2(0,\eta;H^1(\Om))},\quad\forall\,w\in X_\eta,
\end{equation}
that is, $\cL:L^2(0,\eta;H^1(\Om))\to X_\eta$ is bounded. Since the embedding
$X_\eta\to L^2(0,\eta;H^1(\Om))$ is compact, we immediately obtain the
compactness of the operator $\cL:X_\eta\to X_\eta$. On the other hand, by
$1>\al_1>\al_2>\cdots>\al_m>0$, we see
\begin{equation}\label{esti-Kj}
\|\cK_jw\|_{X_\eta}\le C\|w\|_{H^{\al_j}(0,\eta;L^2(\Om))}\le
C\|w\|_{H^{\al_2}(0,\eta;L^2(\Om))},\quad\forall\,w\in X_\eta,\ j=2,\ldots,m,
\end{equation}
where the constant $C>0$ is independent of $\eta\in(0,T)$ (see
\cite[p.434]{SY11a}). Meanwhile, the embedding
$X_\eta\to H^{\al_2}(0,\eta;L^2(\Om))$ is compact (see Temam
\cite[Chapter III, \S2]{T79}, or one can prove directly similarly to
Baumeister \cite[Chapter 5]{B87}), which yields the compactness of
$\cK=\sum_{j=2}^mq_j\cK_j:X_\eta\to X_\eta$ and thus the compactness of
$\cK-\cL:X_\eta\to X_\eta$.

Now we attempt to verify that $1$ is not an eigenvalue of $\cK-\cL$, that is,
$(\cK-\cL)w=w$ in $X_\eta$ implies $w=0$. First we prove
\begin{equation}\label{esti-partial_beta}
\|\pa_t^\be g\|_{L^2(0,\eta)}
\le C\,\eta^{\al_1-\be}\|\pa_t^{\al_1}g\|_{L^2(0,\eta)},
\quad\forall\,\be\in[0,\al_1),\ \forall\,g\in\cR(J_{0+}^{\al_1}).
\end{equation}
Indeed, since $J_{0+}^\ga$ is defined by the fractional power for $\ga\in\BR$,
if follows that (see Pazy \cite[Theorem 6.8]{P83})
\[
J_{0+}^{-\be}g=J_{0+}^{\al_1-\be}(J_{0+}^{-\al_1}g),
\quad g\in\cR(J_{0+}^{\al_1})
\]
and thus
\[
\|\pa_t^\be g\|_{L^2(0,\eta)}=\|J_{0+}^{-\be}g\|_{L^2(0,\eta)}
=\|J_{0+}^{\al_1-\be}(J_{0+}^{-\al_1}g)\|_{L^2(0,\eta)},
\quad g\in\cR(J_{0+}^{\al_1}).
\]
On the other hand, by Young's inequality, there holds for
$g\in\cR(J_{0+}^{\al_1})\subset H^{\al_1}(0,\eta)$ that
\begin{align*}
\|J_{0+}^{\al_1-\be}g\|_{L^2(0,\eta)} & =\f1{\Ga(\al_1-\be)}
\left\|\int^t_0(t-\tau)^{(\al_1-\be)-1}g(\tau)\,\rd\tau\right\|_{L^2(0,\eta)}\\
& \le\f1{\Ga(\al_1-\be)}\int_0^\eta t^{(\al_1-\be)-1}\,\rd t
\left(\int_0^\eta|g(t)|^2\,\rd t\right)^{1/2}
=\f{\eta^{\al_1-\be}}{\Ga(\al_1-\be+1)}\|g\|_{L^2(0,\eta)},
\end{align*}
implying
\[
\|\pa_t^\be g\|_{L^2(0,\eta)}
\le C\,\eta^{\al_1-\be}\|J_{0+}^{-\al_1}g\|_{L^2(0,\eta)}
=C\,\eta^{\al_1-\be}\|\pa_t^{\al_1}g\|_{L^2(0,\eta)}
\]
or equivalently \eqref{esti-partial_beta}. Using \eqref{esti-Kj} and
\eqref{esti-partial_beta}, we estimate
\begin{align}
\|\cK_jw\|_{X_\eta} & \le C\|\pa_t^{\al_j}w\|_{L^2(\Om\times(0,\eta))}
\le C\,\eta^{\al_1-\al_j}\|\pa_t^{\al_1}w\|_{L^2(\Om\times(0,\eta))}\nonumber\\
& \le C\,\eta^{\al_1-\al_2}\|w\|_{H^{\al_1}(0,\eta;L^2(\Om))}
\le C\,\eta^{\al_1-\al_2}\|w\|_{X_\eta},
\quad\forall\,w\in X_\eta,\ j=2,\ldots,m.\label{eq-est-Kj}
\end{align}
Especially, taking $\be=0$ in \eqref{esti-partial_beta}, we obtain
\[
\|w\|_{L^2(\Om\times(0,\eta))}
\le C\,\eta^{\al_1}\|\pa_t^{\al_1}w\|_{L^2(\Om\times(0,\eta))}
\le C\,\eta^{\al_1}\|w\|_{X_\eta},\quad\forall\,w\in X_\eta.
\]
Applying the above estimate and the interpolation inequality to
\eqref{esti-Lu}, we see that for any $\ep>0$, there exists a constant
$C_\ep>0$ such that
\begin{align*}
\|\cL w\|_{X_\eta} & \le C\|w\|_{L^2(0,\eta;H^1(\Om))}
\le\ep\|w\|_{L^2(0,\eta;H^2(\Om))}+C_\ep\|w\|_{L^2(\Om\times(0,\eta))}\\
& \le(\ep+CC_\ep\,\eta^{\al_1})\|w\|_{X_\eta},\quad\forall\,w\in X_\eta.
\end{align*}
This, together with the estimate \eqref{eq-est-Kj}, implies
\begin{align*}
\|(\cK-\cL)w\|_{X_\eta}
& \le\sum_{j=2}^mq_j\|\cK_jw\|_{X_\eta}+\|\cL w\|_{X_\eta}
\le\left(C\,\eta^{\al_1-\al_2}+\ep+CC_\ep\,\eta^{\al_1}\right)\|w\|_{X_\eta},
\quad\forall\,w\in X_\eta.
\end{align*}
Fixing $0<\ep<1$ arbitrarily, we can choose a sufficiently small $\eta_\ep>0$
so that
\[
C\,\eta_\ep^{\al_1-\al_2}+\ep+CC_\ep\,\eta_\ep^{\al_1}<1.
\]
Consequently, if $w =(\cK-\cL)w$ in $X_{\eta_\ep}$, then the only possibility
is $w=0$ in $\Om\times(0,\eta_\ep)$, indicating that $1$ is not an eigenvalue
of $\cK-\cL$ on $X_{\eta_\ep}$.

In the final step, we continue this argument over $\eta_\ep$ to show that
$w=(\cK-\cL)w$ in
$X_{2\eta_\ep}=H^{\al_1}(0,2\eta_\ep;L^2(\Om))\cap L^2(0,2\eta_\ep;H^2(\Om))$
implies $w=0$ in $\Om\times(0,2\eta_\ep)$. To this end, we investigate
$\wt w(\,\cdot\,,t):=w(\,\cdot\,,t+\eta_\ep)$. Now that $w=0$ in
$\Om\times(0,\eta_\ep)$, formally we calculate
\begin{align*}
\pa_t^{\al_j}w(\,\cdot\,,t+\eta_\ep)
& =\f1{\Ga(1-\al_j)}\int_{\eta_\ep}^{t+\eta_\ep}
\f{\pa_\tau w(\,\cdot\,,\tau)}{(t+\eta_\ep-\tau)^{\al_j}}\,\rd\tau
=\f1{\Ga(1-\al_j)}\int_0^t
\f{\pa_\xi w(\,\cdot\,,\xi+\eta_\ep)}{(t-\xi)^{\al_j}}\,\rd\xi\\
& =\f1{\Ga(1-\al_j)}\int_0^t
\f{\pa_\xi\wt w(\,\cdot\,,\xi)}{(t-\xi)^{\al_j}}\,\rd\xi
=\pa_t^{\al_j}\wt w(\,\cdot\,,t),\quad t>0.
\end{align*}
By the definition of $\cK_j$, we employ again the fact that $w=0$ in
$\Om\times(0,\eta_\ep)$ to deduce
\begin{align*}
\cK_jw(\,\cdot\,,t+\eta_\ep) & =\int_0^{t+\eta_\ep}
S_{\al_1}(t+\eta_\ep-\tau)\pa_\tau^{\al_j}w(\,\cdot\,,\tau)\,\rd\tau
=\int_{\eta_\ep}^{t+\eta_\ep}
S_{\al_1}(t+\eta_\ep-\tau)\pa_\tau^{\al_j}w(\,\cdot\,,\tau)\,\rd\tau\\
& =\int_0^tS_{\al_1}(t-\xi)\pa_\xi^{\al_j}w(\,\cdot\,,\xi+\eta_\ep)\,\rd\xi
=\int_0^tS_{\al_1}(t-\xi)\pa_\xi^{\al_j}\wt w(\,\cdot\,,\xi)\,\rd\xi\\
& =\cK_j\wt w(\,\cdot\,,t),\quad t>0,\ j=2,\ldots,m.
\end{align*}
Similarly, we obtain
\[
\cL w(\,\cdot\,,t+\eta_\ep)=\cL\wt w(\,\cdot\,,t),\quad t>0.
\]
Eventually, we collect the above equalities to conclude
\begin{align*}
\wt w(\,\cdot\,,t) & =w(\,\cdot\,,t+\eta_\ep)
=\sum_{j=2}^mq_j\cK_jw(\,\cdot\,,t+\eta_\ep)-\cL w(\,\cdot\,,t+\eta_\ep)\\
& =\sum_{j=2}^mq_j\cK_j\wt w(\,\cdot\,,t)-\cL\wt w(\,\cdot\,,t)
=(\cK-\cL)\wt w(\,\cdot\,,t),\quad t>0.
\end{align*}
Therefore, the same argument for $w\in X_\eta$ immediately yields $\wt w=0$ in
$\Om\times(0,\eta_\ep)$ and thus $w=0$ in $\Om\times(0,2\eta_\ep)$. Since the
step size $\eta_\ep$ is a positive constant, we can repeat the same argument
finite times to reach the conclusion that $w=(\cK-\cL)w$ in
$X_T=H^{\al_1}(0,T;L^2(\Om))\cap L^2(0,T;H^2(\Om))$ implies $w=0$ in
$Q=\Om\times(0,T)$.

Consequently, by the Fredholm alternative, we complete the proof of Lemma
\ref{lem-fp-weak}.
\end{proof}


{\bf Acknowledgement}\ \ The work was supported by A3 Foresight Program
``Modeling and Computation of Applied Inverse Problems'', Japan Society of the
Promotion of Science (JSPS). The first author is supported by self-determined
research funds of CCNU from the colleges' basic research and operation of MOE
(No.\! CCNU14A05039), National Natural Science Foundation of China (Nos.\!
11326233, 11401241 and 11571265). The second author is partially supported by
the Program for Leading Graduate Schools, MEXT, Japan. The other authors are
partially supported by Grant-in-Aid for Scientific Research (S) 15H05740, JSPS.


\begin{thebibliography}{99}

\bibitem{AG92}
Adams E E and Gelhar L W 1992 Field study of dispersion in a heterogeneous
aquifer: 2. Spatial moments analysis {\it Water Resour. Res.} {\bf28} 3293--307

\bibitem{A75}
Adams R A 1975 {\it Sobolev Spaces} (New York: Academic Press)

\bibitem{B87}
Baumeister J 1987 {\it Stable Solution of Inverse Problems} (Braunschweig:
Vieweg)

\bibitem{CLN13}
Cheng J, Lin C-L and Nakamura G 2013 Unique continuation property for the
anomalous diffusion and its application {\it J. Differential Equations}
{\bf254} 3715--28

\bibitem{CNYY09}
Cheng J, Nakagawa J, Yamamoto M and Yamazaki T 2009 Uniqueness in an inverse
problem for a one-dimensional fractional diffusion equation {\it Inverse
Problems} {\bf25} 115002

\bibitem{DDM04}
Daubechies I, Defrise M and De Mol C 2004 An iterative thresholding algorithm
for linear inverse problems {\it Comm. Pure Appl. Math.} {\bf57} 1413--57

\bibitem{GLY15}
Gorenflo R, Luchko Y and Yamamoto M 2015 Time-fractional diffusion equation in
the fractional Sobolev spaces {\it Frac. Calc. Appl. Anal.} {\bf18} 799--820

\bibitem{F14}
Fujishiro K 2014 Approximate controllability for fractional diffusion equations
by Dirichlet boundary control arXiv:1404.0207v3

\bibitem{HH98}
Hatano Y and Hatano N 1998 Dispersive transport of ions in column experiments:
an explanation of long-tailed profiles {\it Water Resour. Res.} {\bf34}
1027--33

\bibitem{JLLZ15}
Jin B, Lazarov R, Liu Y and Zhou Z 2015 The Galerkin finite element method for
a multi-term time-fractional diffusion equation {\it J. Comput. Phys.} {\bf281}
825--43

\bibitem{JLZ13}
Jin B, Lazarov R and Zhou Z 2013 Error estimates for a semidiscrete finite
element method for fractional order parabolic equations {\it SIAM J. Numer.
Anal.} {\bf51} 445--66

\bibitem{LLY15}
Li Z, Liu Y and Yamamoto M 2015 Initial-boundary value problems for multi-term
time-fractional diffusion equations with positive constant coefficients {\it
Appl. Math. Comput.} {\bf257} 381--97

\bibitem{LY13}
Li Z and Yamamoto M 2013 Initial-boundary value problems for linear diffusion
equation with multiple time-fractional derivatives arXiv:1306.2778v2

\bibitem{LIY16}
Li Z, Imanuvilov O and Yamamoto 2016 Uniqueness in inverse boundary value
problems for fractional diffusion equations {\it Inverse Problems} {\bf32}
015004

\bibitem{LY15}
Li Z and Yamamoto M 2015 Uniqueness for inverse problems of determining orders
of multi-term time-fractional derivatives of diffusion equation {\it Appl.
Anal.} {\bf94} 570--9

\bibitem{LZJY13}
Li G, Zhang D, Jia X and Yamamoto M 2013 Simultaneous inversion for the
space-dependent diffusion coefficient and the fractional order in the
time-fractional diffusion equation {\it Inverse Problems} {\bf29} 065014

\bibitem{LN16}
Lin C-L and Nakamura G 2016 Unique continuation property for anomalous slow
diffusion equation {\it Commun. Partial Diff. Eqns} at press

\bibitem{Lin07}
Lin Y and Xu C 2007 Finite difference/spectral approximations for the
time-fractional diffusion equation {\it Appl. Math. Comput.} {\bf225} 1533--52

\bibitem{L15}
Liu Y 2015 Strong maximum principle for multi-term time-fractional diffusion
equations and its application to an inverse source problem arXiv:1510.06878

\bibitem{LJY15}
Liu Y, Jiang D and Yamamoto M 2015 Inverse source problem for a double
hyperbolic equation describing the three-dimensional time cone model {\it SIAM
J. Appl. Math.} {\bf75} 2610--35

\bibitem{LRY16}
Liu Y, Rundell W and Yamamoto M 2016 Strong maximum principle for fractional
diffusion equations and an application to an inverse source problem, {\it Frac.
Calc. Appl. Anal.} (accepted)

\bibitem{L10}
Luchko Y 2010 Some uniqueness and existence results for the
initial-boundary-value problems for the generalized time-fractional diffusion
equation {\it Comput. Math. Appl.} {\bf59} 1766--72

\bibitem{L11}
Luchko Y 2011 Initial-boundary-value problems for the generalized multi-term
time-fractional diffusion equation {\it J. Math. Anal. Appl.} {\bf374} 538--48

\bibitem{MY13}
Miller L and Yamamoto M 2013 Coefficient inverse problem for a fractional
diffusion equation {\it Inverse Problems} {\bf29} 075013

\bibitem{P83}
Pazy A 1983 {\it Semigroups of Linear Operators and Applications to Partial
Differential Equations} (Berlin: Springer)

\bibitem{P99}
Podlubny I 1999 {\it Fractional Differential Equations} (San Diego: Academic)

\bibitem{SY11a}
Sakamoto K and Yamamoto M 2011 Initial value/boundary value problems for
fractional diffusion-wave equations and applications to some inverse problems
{\it J. Math. Anal. Appl.} {\bf382} 426--47

\bibitem{SY11b}
Sakamoto K and Yamamoto M 2011 Inverse source problem with a final
overdetermination for a fractional diffusion equation {\it Math. Control Relat.
Fields} {\bf1} 509--18

\bibitem{SS87}
Saut J C and Scheurer B 1987 Unique continuation for some evolution equations
{\it J. Differential Equations} {\bf66} 118--39

\bibitem{T79}
Temam R 1977 {\it Navier-Stokes Equations: Theory and Numerical Analysis}
(Amsterdam: North-Holland)

\bibitem{XCY11}
Xu X, Cheng J and Yamamoto M 2011 Carleman estimate for a fractional diffusion
equation with half order and application {\it Appl. Anal.} {\bf90} 1355--71

\bibitem{YZ12}
Yamamoto M and Zhang Y 2012 Conditional stability in determining a zeroth-order
coefficient in a half-order fractional diffusion equation by a Carleman
estimate {\it Inverse Problems} {\bf28} 105010

\bibitem{Z16}
Zhang Z 2016 An undetermined coefficient problem for a fractional diffusion
equation {\it Inverse Problems} {\bf32} 015011

\bibitem{ZX11}
Zhang Y and Xu X 2011 Inverse source problem for a fractional diffusion
equation {\it Inverse Problems} {\bf27} 035010

\end{thebibliography}
\end{document}